\documentclass[english,12pt]{article} 
\pdfoutput=1
\usepackage{authblk}
\usepackage[top=2cm,bottom=2cm,left=2cm,right=2cm]{geometry}
\usepackage{fontenc}
\usepackage[utf8]{inputenc}
\usepackage{afterpage}
\usepackage[dvipsnames]{xcolor}
\usepackage{newtxtext,commath}
\usepackage{amscd, amsmath, amsthm, amssymb, amsfonts}
\usepackage{color}
\usepackage{enumerate}
\usepackage{graphicx}
\usepackage{ragged2e}
\usepackage{hyperref}
\usepackage{times}
\numberwithin{equation}{section}
\usepackage{comment}

\newtheorem{theorem}{Theorem}[section]
\newtheorem{corollary}[theorem]{Corollary}
\newtheorem{lemma}[theorem]{Lemma}

\newtheorem{definition}[theorem]{Definition}
\newtheorem{remark}[theorem]{Remark}
\newtheorem{prop}[theorem]{Proposition}

\numberwithin{equation}{section}

\newcommand{\R}{\mathbb{R}}

\newcommand{\eps}{\varepsilon}
\newcommand{\Om}{\Omega}
\newcommand{\skelet}{\mathcal{S}(\Omega)}
\newcommand{\ridge}{\mathcal{R}(\Omega)}
\newcommand{\prt}{\partial}

\newcommand{\lmb}{\lambda}

\newcommand{\dvg}{\text{div}}

\newcommand{\hess}{\text{H}}
\newcommand{\dist}{\text{dist}}
\newcommand{\W}{W^{2,p}_0(\Om)}
\newcommand{\C}{C_c^\infty(\Om)}
\newcommand{\card}{\text{card}}
\newcommand{\m}{\text{max}}
\newcommand{\n}{\vec{n}}

\renewenvironment{thebibliography}[1]{
  \begin{oldthebibliography}{#1}
    \setlength{\itemsep}{0.2em}
    \setlength{\parskip}{0.2em}
}
{
  \end{oldthebibliography}
}

\title{Sharp second order inequalities with distance function to the boundary and applications to a p-Biharmonic singular problem}
\author[1, 2]{Cristian Cazacu} 
\author[1, 2]{Teodor Rugină}
\affil[1]{Faculty of Mathematics and Computer Science, University of Bucharest,
		14 Academiei Street, 
		010014 Bucharest, Romania.\newline Emails: cristian.cazacu@fmi.unibuc.ro, teorugina@yahoo.com}
\affil[2]{Gheorghe Mihoc-Caius Iacob Institute of Mathematical
	Statistics and Applied Mathematics of the Romanian Academy\\
	050711 Bucharest, Romania}
    \date{}

\begin{document}
\maketitle\raggedright

\begin{abstract}
In this paper, we prove generalizations to the $L^p-$setting of the Hardy-Rellich inequalities on domains of $\R^N$ with singularity given by the distance function to the boundary. The inequalities we obtain are either sharp in bounded domains, where we provide concrete minimizing sequences, or give a new bound for the sharp constant, while also depending on the geometric properties of the domain and its boundary. We also give applications to the existence and non-existence of solutions for a singular problem using variational methods and a Pohozaev identity.
\end{abstract}

\section{Introduction}
 Functional inequalities with first and second order derivatives are fundamental in mathematical analysis, theory of function spaces, and have crucial importance, especially to the study of second and fourth order PDEs or eventually to second order PDEs for systems where the unknowns are curl-free or divergence-free vector fields. 

Let us consider $N\geq 2$, $\Omega\subset\R^N$ be an open domain with sufficiently smooth boundary $\prt\Om$ and the potentials $V, W, Z$, respectively, which are supposed to be locally integrable, non-trivial and nonnegative. Typically, a first order inequality of the form 
\begin{equation}\label{HI}
\int_{\Omega} |\nabla u|^p dx \geq \lambda \int_\Omega V(x)|u|^p dx, \quad u\in C_c^\infty(\Omega),
\end{equation}
is referred to as a Hardy-type inequality. A second order inequality of the form 
\begin{equation}\label{R}
\int_{\Omega} |\Delta u|^p dx \geq \lambda \int_\Omega W(x)|u|^p dx, \quad u\in C_c^\infty(\Omega), 
\end{equation}
is referred to as a Rellich-type inequality. The Hardy-Rellich inequality combines features of both previous inequalities:
\begin{equation}\label{HR}
\int_{\Omega} |\Delta u|^p dx \geq \lambda \int_\Omega Z(x)|\nabla u|^p dx, \quad u\in C_c^\infty(\Omega). 
\end{equation}

  In \eqref{HI}-\eqref{HR} $\lambda$'s denote positive constants depending in general on the dimension $N$, $p$, $\Omega$.
We denote by $\lambda^\sharp(p, \Omega, V)$, $\lambda^\sharp(p, \Omega, W)$ and $\lambda^\sharp(p, \Omega, Z)$,  the sharp constants in  \eqref{HI}, \eqref{R} and \eqref{HR} respectively, meaning  the greatest constants $\lambda$'s for which \eqref{HI}, \eqref{R} and \eqref{HR} hold.
  
A special attention is drawn when the potentials $V, W, Z$ become unbounded (equiv. singular), in which case they become more interesting from the mathematical point of view, aspects correlated also with what physics dictates, for instance, in quantum mechanics, where Schrodinger operators with Couloumb type potentials appear. Particularly, when singular potentials are involved in the analysis, there are various questions to be addressed, such as: the existence of Hardy and Rellich type inequalities, quantification of the sharp constants, existence and determination of extreme functions, improvements by adding remainder terms, etc.

Next, we particularly refer to potentials involving the distance function to the boundary $\delta_\Omega: \overline{\Omega}\rightarrow \R_{+}$  
defined  by  
 $   \delta_\Om(x) : = \text{dist}(x,\prt\Om) := \inf_{y\in\prt\Om} \abs{x-y} \notag$
and, when there is no confusion, we will write $\delta$ instead of $\delta_\Om$. 

The most studied cases in the literature so far concern with potentials involving one singular potential or distance function to the boundary (such as $V(x)\sim \frac{1}{|x|^p},\frac{1}{\delta^p(x)}$,  $W(x)\sim \frac{1}{|x|^{2p}}, \frac{1}{\delta(x)^{2p}}$, $Z(x)\sim \frac{1}{|x|^{p}}$). The special cases of \eqref{HI}-\eqref{HR} for $p=2$ have been intensively studied due to their Hilbertian structure. For all these cases, the sharp constants are known in many situations as follows. 
 \begin{enumerate}[I.]
 \item The case $V(x)=\frac{1}{|x|^p}, 1\leq p<N$: 
 \begin{itemize}
     \item If $\Om=\R^N$ or $0\in\Om$ then $\lambda^\sharp(p, \Omega, V)=\left(\frac{N-p}{p}\right)^p$; \;\;(cf. \cite{hardy}; for a simple proof, see e.g. \cite{azorero-alonso})
     \item If $\Om\subseteq\R^N_+$ and $0\in \prt\Om$ then $\lambda^\sharp(2,\Om,V)=\frac{N^2}{4}$.  \;\; (cf. \cite{cazacu-boundary}; for further extensions for any $p>1$ see, e.g., \cite{mmfall}) 
 \end{itemize}
 \item The case $V(x)=\frac{1}{\delta^p(x)}$, $1<p<\infty$: 
 \begin{itemize}
     \item For all $\Omega$ bounded it holds $\lambda^\sharp(p, \Omega, V)\leq \left(\frac{p-1}{p}\right)^p$; \;\; (cf. \cite{pinchover})
     \item If $\Omega$ is convex then  $\lambda^\sharp(p, \Omega, V)=\left(\frac{p-1}{p}\right)^p$; \;\; (cf. \cite{sobolevskii} for $N=2$, \cite{pinchover})
     \item There are non-convex domains satisfying $ \lambda^\sharp(p, \Omega, V)=\left(\frac{p-1}{p}\right)^p$; \;\;(e.g. weakly mean convex domains, cf. \cite{lewis})
     \item There exist non-convex domains $\Omega$ for which $\lambda^\sharp(p, \Omega, V)< \left(\frac{p-1}{p}\right)^p$; \;\; (cf. \cite[Ex. 1]{pinchover} - for $\Om=\R^N\setminus\{0\}$ it holds $\lambda^\sharp(p, \Omega, V)=\left|\frac{N-p}{p}\right|^p$ ) 
     \item If $\Om$ is an arbitrary planar, simply connected domain, then $\lambda^\sharp(2,\Om,V) \geq \frac{1}{16}$. \; (cf. \cite{ancona})
 \end{itemize}
 \item The case $W(x)=\frac{1}{|x|^{2p}}$ and $0\in\Om$,
 \begin{itemize}
     \item If $p=2$ and $N\geq 5$, then $\lambda^\sharp(2,\Om,W)=\left(\frac{N(N-4)}{4}\right)^2$; \; (known as Rellich inequality, cf. \cite{rellich}); it also holds when $N\in \{2, 3\}$ for functions $u\in C_c^\infty(\Omega\setminus\{0\})$ 
     \item If $1<p<\frac{N}{2}$, then $\lambda^\sharp(p,\Om,W)=\left(\frac{N(N-2p)(p-1)}{p^2}\right)^p$; \;\big(cf. \cite{davies-hinz} - for $u\in C^\infty_c(\Om\setminus\{0\})$\big) 
     \item If $1<p<\frac{N}{2}$ and $\Om=\R^N$, then $\lambda^\sharp(p,\R^N,W)=\left(\frac{N(N-2p)(p-1)}{p^2}\right)^p$; (cf. \cite{balinsky1}, Cor. 6.3.5)
     \item It seems that the case $0\in\prt\Om$ has not been studied.
 \end{itemize}
 \item The case $W(x)=\frac{1}{\delta(x)^{2p}}$, $1<p<\infty$:
     \begin{itemize}
         \item If $p=2$ and $\Om$ is convex then $\lambda^\sharp(2,\Om,W)=\frac{9}{16}$; \; (cf. \cite{owen})
         \item If $p>1$ then $\lambda^\sharp(p,\Om,W)\leq \left(\frac{(2p-1)(p-1)}{p^2}\right)^p$; \; (cf. \cite{barbatis3}, \cite{barbatis}) -the exact value of $\lmb^\sharp$ in this case seems to be known only for $p=2$;
         \item If $p>1$, $p\neq 2$ and $\Om$ is convex then $\lambda^\sharp(p,\Om,W) \leq B_{N,2p} N^{-d}\cot^{-2p}\left(\frac{\pi}{2p^\m}\right)\left(\frac{(2p-1)(p-1)}{p^2}\right)^p$, \;(cf. \cite{edmunds})\\
         where $d=\left\{\begin{array}{cc} 
    2,  &   1<p<2 \\ 
    \frac{2p}{p'},  &  2<p<\infty  
\end{array} \right.$, $p^\m=\max\{p,\frac{p}{p-1}\}$, $B_{N,2p}=\frac{\Gamma\left(\frac{1+2p}{2}\right)\Gamma\left(\frac{N}{2}\right)}{\sqrt{\pi}\;\Gamma\left(\frac{N+2p}{2}\right)}$, $\Gamma$ being the standard Euler $\Gamma$ function.
     \end{itemize}
\item The case $Z(x)=\frac{1}{|x|^p}$:
    \begin{itemize}
        \item If $p=2$ and $\Om=\R^N$ then $\lambda^\sharp(2,\R^N,Z) = \left\{\begin{array}{cc} 
    \frac{N^2}{4},  &  N\geq 5, \\ 
    3,  &  N=4,  \\ 
    \frac{25}{36},  &  N=3,  \\
    0,  &  N=2,\\
    \frac{1}{4},  &  N=1;   
\end{array} \right.$   \big(cf. \cite{cazacuHR} - for $N\geq 3$; \cite{cassano} - for $N=1,2$ and $u\in C^\infty_c(\R^N\setminus\{0\})$\big)
        \item Any $p>1$: seems to be unsolved.
    \end{itemize}
\item The case $Z(x)=\frac{1}{\delta^p(x)}$, $1<p<\infty$:
    \begin{itemize}
        \item If $p=2$, $\Om$ convex and $\delta(x)$ bounded in $\Om$ then $\lambda^\sharp(2,\Om,Z)=\frac{1}{4}$; \; (cf. \cite{barbatis})
        \item Any $p>1$: seems to be unsolved.
    \end{itemize}
 \end{enumerate}

In this paper, we focus on the unsolved problems of case VI above. That is, we study second order $L^p$ functional inequalities with singular potentials depending on the distance function to the boundary, in the spirit of \eqref{HR} with the toy-model which corresponds to $Z(x)=\frac{1}{\delta^p(x)}$, with or without reminder terms.  
\paragraph{Well-known improvements with reminder terms.} Hardy inequalities such as \eqref{HI} with $V(x)\sim \frac{1}{\delta^p(x)}$ and reminder terms have been intensively studied so far.  In the case $p=2$, in the pioneering  paper \cite{brezis} the authors  proved that there exists a real constant $\mu=\mu(\Omega)>-\infty$ such that
\begin{equation}\label{ec:Hardy_reminder}
    \int_\Om \abs{\nabla u}^2 dx \geq \frac{1}{4} \int_\Om \frac{\abs{u}^2}{\delta^2(x)} dx + \mu \int_\Om \abs{u}^2 dx, \;\;\;\forall\; u\in C^\infty_c(\Om),     
\end{equation}
where the constant $\frac{1}{4}$ is sharp in any domain modulo the $L^2$ term $\mu \|u\|_{L^2(\Omega)}^2$ ($\mu$ could be also negative). However, they showed that for convex domains the remainder term in \eqref{ec:Hardy_reminder} is positive and more precisely \eqref{ec:Hardy_reminder} is valid with $\mu(\Om)= \frac{1}{4 \textrm{diam}^2(\Om)}$, where $\textrm{diam}(\Omega)=\sup_{x, y\in \Omega}|x-y|$ is the diameter of $\Omega$. This constant was later improved in \cite{hoffman} to $\mu(\Om) = \frac{c(N)}{|\Om|^{2/N}}$, where $c(N)=\frac{N^{(N-2)/N}|\mathbb{S}^{N-1}|^{2/N}}{4}$ ($\mathbb{S}^{N-1}$ denotes the unit sphere in $\mathbf{\mathbb{R}^N}$ and $|\mathbb{S}^{N-1}|$ states for the $(N-1)$ dimensional Hausdorff measure in $\mathbb{R}^N$). An extension of \eqref{ec:Hardy_reminder} to the $L^p$ setting was proved in \cite{tidblom} for convex domains:
\begin{equation}\label{ec:Hardy_reminder_Lp}
    \int_\Om \abs{\nabla u}^p dx \geq \left(\frac{p-1}{p}\right)^p \int_\Om \frac{\abs{u}^p}{\delta^p(x)} dx + \frac{c(N,p)}{|\Om|^{p/N}} \int_\Om \abs{u}^p dx, \;\;\;\forall\; u\in C^\infty_c(\Om),     
\end{equation}
where $c(N,p) = \frac{(p-1)^{p+1}}{p^p}\left( \frac{|\mathbb{S}^{N-1}|}{N}\right)^\frac{p}{N} \frac{\sqrt{\pi}\Gamma\left(\frac{N+p}{2}\right)}{\Gamma\left(\frac{p+1}{2}\right)\Gamma\left(\frac{N}{2}\right)}$. They also prove a counterpart of \eqref{ec:Hardy_reminder_Lp} for $\Om$ not necessarily convex by replacing the potential $1/\delta^p(x)$ with a weaker singular potential on the boundary depending on the weaker distance function $\rho_\nu(x)=\min\{s>0 \ |\ x+ s\nu\not\in \Omega\}$, $\nu \in S^{N-1}$. 

Another geometric inequality in $L^p$ with remainder term was proved in \cite{balinsky2} (or \cite{balinsky1} for a more detailed version), where it is proved (among others) that
 \begin{equation}   \label{ec:Hardy-Dist-rmnd}
        \int_\Om \abs{\nabla u}^p dx \geq \Big(\frac{p-1}{p} \Big)^p \int_\Om \frac{\abs{u}
        ^p}{\delta^p(x)} dx + \Big(\frac{p-1}{p} \Big)^{p-1} \int_\Om \frac{\abs{u}
        ^p}{\delta^{p-1}(x)} (-\Delta\delta) dx,  \;\;\; \forall u\in C^\infty_c(\Om\setminus\ridge),
\end{equation}
where $\ridge\subset \Omega$ is the ridge set of $\Om$ (see Section 3). This restriction is imposed since for domains with sufficiently smooth boundaries, say $\prt\Om \in C^2$, the function $\delta$ is not necessarily smooth on $\ridge$ where $\Delta \delta$ does not exist in the strong sense. When comparing \eqref{ec:Hardy_reminder_Lp} and \eqref{ec:Hardy-Dist-rmnd} we notice that the remainder term in the second inequality also contains a singularity, but is altogether less singular than the first term. It is also interesting that the positivity of the remainder in \eqref{ec:Hardy-Dist-rmnd} depends on a geometric condition on the domain $\Om$, more precisely $-\Delta\delta \geq 0$, which is true in convex domains, but not exclusively, for instance, the case of a torus (cf. \cite{balinsky2}). The authors in \cite{lewis} prove that the condition $-\Delta\delta\geq 0$ is equivalent to the weakly mean convexity of the domain $\Om$; these are defined as $C^2-$domains with non-negative mean curvature in any point on $\prt\Om$. Such domains will also be of great importance to us, so we shall explore their properties further in Section 2, Remark \ref{ec:remark-convexity}. 

To our knowledge, even less is known about the inequality \eqref{R} with potentials involving distance to the boundary compared to the case of \eqref{HI} (we refer to cases II and IV); the same for the counterpart with reminder terms. A notable result for $p=2$ and any convex domain $\Om\subsetneq \R^N$ with finite inradius $\delta_0:=\sup_\Om \delta$ states that (cf. \cite[Prop. 6.2.9]{balinsky1})
\begin{equation}\label{HRR}
    \int_\Omega \abs{\Delta u}^2 dx \geq \frac{9}{16} \int_\Om \frac{\abs{u}^2}{\delta^4(x)} dx +\frac{\lmb_0^2}{4\delta_0^2}\int_\Om \frac{\abs{u}^2}{\delta^2(x)} dx + \frac{\lmb_0^4}{\delta_0^4}\int_\Om \abs{u}^2 dx,
\end{equation}
where $\lmb_0$ is the first zero in $(0,\infty)$ of $J_0(x)-2x J_1(x)$ and $J_0$ and $J_1$ are Bessel functions (e.g., a comprehensive study of the Bessel functions can be found in \cite{bowman}). \\
For the inequality \eqref{HR} with $Z(x)=\frac{1}{\delta^p(x)}$, we refer again to the reference \cite{barbatis}, where positive logarithmic-type singular remainder terms are obtained. As a consequence, when $\Om$ is a convex bounded domain, then the next optimal inequality holds:
\begin{equation}  \label{ec:hardy-rellich-dist}
    \int_{\Om} \abs{\Delta u}^2 dx \geq \frac{1}{4} \int_{\Om} \frac{\abs{\nabla u}^2}{\delta^2(x)} dx, \;\;\;\forall\; u\in C^\infty_c(\Om).  
\end{equation}
This can be seen as a Hardy-type inequality with distance function to the boundary applied for vector fields $\vec{U}$ which are potential fields, i.e. $\vec{U}=\nabla u$, where $u:\Omega\rightarrow \R$ is a scalar field.\\
It is interesting to remark that when we pass from the Hardy inequality \eqref{HI} to the Hardy-Rellich inequality \eqref{HR} with singular potential given by the distance to the origin (i.e., cases I and V for $p=2$), we notice an improvement of the constant. These facts are well-covered and further explored in the works \cite{cazacuHR}, \cite{cazacu}, \cite{tertikas}. This phenomenon doesn't seem to appear when we consider inequalities involving the distance to the boundary function, at least not in the case $p=2$, where we can see above in \eqref{ec:hardy-rellich-dist} that the optimal constant remains the same as in the Hardy inequality (case II, for $p=2$). Extended literature on the subject can be found in \cite{avkhadiev1}, \cite{kombe}, \cite{lam1}, \cite{lam2}. We were unable to find more improvements on inequalities of \eqref{ec:hardy-rellich-dist}-type with remainder terms for $p=2$ or $p\neq 2$, which encouraged us to look more into the subject. \\
Our work aims to find an $L^p-$version of \eqref{ec:hardy-rellich-dist}, the sharp constant in particular domains, such as convex domains, and develop improved inequalities in the spirit of \eqref{ec:Hardy-Dist-rmnd}. We find interesting geometric restrictions on the domain in which our inequalities are sharp, depending on the curvature of $\prt\Om$. To the best of our knowledge, based on our brief analysis of the existing literature - outlined in cases I-VI above and further extended with \eqref{ec:Hardy_reminder}-\eqref{ec:hardy-rellich-dist} - our results seem to be new, even in the case where $p=2$.\\

\section{Main results}  
In the following, we denote by $\hess u$ the Hessian matrix of $u$ with its norm given by 
\begin{equation}   \label{ec:hessiannorm}
    \abs{\hess u}=\Big( \sum_{i,j=1}^N u_{x_ix_j}^2 \Big)^\frac{1}{2}.
\end{equation}
In the sequel we define by $\lambda^\sharp_H\left(p,\Om,\frac{1}{\delta^p}\right)$ the sharp constant in the following $L^p$-second order Hardy-Rellich inequality 
\begin{equation}   \label{ec:hessianineq_pure}
        \int_\Om \abs{\hess u}^p dx \geq \lambda^\sharp_H\left(p,\Om,\frac{1}{\delta^p}\right) \int_\Om \frac{\abs{\nabla u}
        ^p}{\delta^p(x)} dx, \;\;\; \forall u\in C^\infty_c(\Om).
    \end{equation}
\begin{theorem}   \label{thm1}
    Let $N\geq 2$, $1<p<\infty$ and $\Om\subset\R^N$ open domain with $C^2$ boundary $\prt\Om$. Then for any $ u\in C^\infty_c(\Om\setminus\ridge)$ it holds:
    \begin{equation}   \label{ec:hessianineq}
        \int_\Om \abs{\hess u}^p dx \geq \Big(\frac{p-1}{p} \Big)^p \int_\Om \frac{\abs{\nabla u}
        ^p}{\delta^p(x)} dx + \Big(\frac{p-1}{p} \Big)^{p-1} \int_\Om \frac{\abs{\nabla u}
        ^p}{\delta^{p-1}(x)} (-\Delta\delta) dx.
    \end{equation}
   \begin{itemize}
\item[i)]  Particularly, if $\Om$ is bounded and $-\Delta\delta \geq 0$ in $\Om\setminus\ridge$, then
    \begin{equation}   \label{ec:hessianineq2}
        \int_\Om \abs{\hess u}^p dx \geq \Big(\frac{p-1}{p} \Big)^p \int_\Om \frac{\abs{\nabla u}
        ^p}{\delta^p(x)} dx,  
    \end{equation}
    and the constant $\lambda_p:=\Big(\frac{p-1}{p} \Big)^p$ is sharp in \eqref{ec:hessianineq2}, i.e. 
    $$\lambda^\sharp_H\left(p,\Om,\frac{1}{\delta^p}\right)=\lambda_p.$$\\
\item[ii)]  Otherwise,  it holds 
    \begin{equation}   \label{ec:HRsharp}
       \max\left\{\lambda_H^\sharp\left(p,\Om,\frac{1}{\delta^p}\right), \lambda^\sharp \left(p,\Om,Z=\frac{1}{\delta^p}\right)\right\}  \leq \lambda_p.
    \end{equation} 
   \end{itemize}
\end{theorem} 
The inequality known in the literature as the Calderón–Zygmund inequality (see e.g. \cite{gilbardtrud}, \cite{pigola}, \cite{metafune}) states that
    if $\Om$ is an open bounded domain in $\R^N$ and $1<p<\infty$, then for any $u\in C_c^\infty(\Om)$ there exists a positive constant $C=C(N,p)$ such that
    \begin{equation}
        \int_\Om \abs{\Delta u}^p dx \geq C \int_\Om \abs{\hess u}^p dx.   \label{eq:CZineq}
    \end{equation}
We denote the best possible constant in \eqref{eq:CZineq} by $C_{CZ,p}$. To our knowledge  the explicit value of $C_{CZ,p}$ is not known, except for the case $p=2$ when $C_{CZ,2}=1$ due to the fact that the two integrals coincide in view of the computation below, based on two integrations by parts:
\begin{align}
    \int_\Om \abs{\Delta u}^2 dx & =  \sum_{i=1}^N\int_\Om u^2_{x_ix_i} dx + 2\sum_{1\leq i<j\leq N} \int_\Om u_{x_ix_i}u_{x_jx_j} dx    \notag\\
    & = \sum_{i=1}^N \int_\Om u^2_{x_ix_i} dx + 2\sum_{1\leq i<j\leq N} \int_\Om u^2_{x_ix_j} dx    \notag\\
    & = \sum_{i=1,j}^N \int_\Om u^2_{x_ix_j} dx = \int_\Om \abs{\hess u}^2 dx.   \label{eq:calculation1}
\end{align}
When $1<p<\infty$, the following estimate from \cite{edmunds} holds for any $u\in C^2_c(\R^N)$ and any $i,j=1,...,N$:
\begin{equation}   \label{xixj-bound}
    \|u_{x_ix_j}\|_{L^p} \leq \cot^2 \left( \frac{\pi}{2p^{\m}} \right) \|\Delta u\|_{L^p},  \;\;\;\text{where} \;\;\; p^{\m}:=\max\{p,p'\}, \;\;\; \frac{1}{p}+\frac{1}{p'}=1.
\end{equation} 
The approach of proving \eqref{xixj-bound} in \cite{edmunds} is based on the explicit computation of the norm of the Riesz transform as a map from $L^p(\R^N)$ to itself. This was done first in \cite{pichorides} on $L^p(\R)$ and extended later in \cite{banuelos} and \cite{iwaniec}. As a consequence of \eqref{xixj-bound}, we can easily deduce the next proposition (see Appendix for a proof). 
\begin{prop}   \label{prop1}
    Let $N\geq 2$, $1<p<\infty$, $p\neq 2$ and $u\in C^2_c(\R^N)$. Then 
    \begin{equation}    \label{hess-lap}
        \int_{\R^N} \abs{\Delta u}^p dx \geq C_{p,N} \int_{\R^N} \abs{\hess u}^p dx,
    \end{equation}
    where $C_{p,N}= N^{-p}  \cot^{-2p} \left( \frac{\pi}{2p^{\m}} \right)$.
\end{prop}
This gives rise to the following lower bound of the sharp constant in \eqref{eq:CZineq}, i.e. 
$C_{p,N}\leq C_{CZ,p}$. However, this lower-bound is not sharp, since it is far worse in the case $p=2$ than the actual constant $C_{CZ,2}=1$. As a trivial consequence of Theorem \ref{thm1} and inequality \eqref{hess-lap} we get 
\begin{corollary} \label{thm2}
    Let $N\geq 2$, $1<p<\infty$ and $\Om\subset\R^N$ open domain with $C^2$ boundary $\prt\Om$ and $u\in C^\infty_c(\Om\setminus\ridge)$. If $p=2$, then
    \begin{equation}    \label{ec:laplacianineq-p=2}
        \int_\Om \abs{\Delta u}^2 dx \geq \frac{1}{4}\int_\Om \frac{\abs{\nabla u}
        ^2}{\delta^2(x)} dx + \frac{1}{2} \int_\Om \frac{\abs{\nabla u}
        ^2}{\delta(x)} (-\Delta\delta) dx,
    \end{equation}
    while if $p\in(1,\infty)\setminus\{2\}$, then  
    \begin{equation}   \label{ec:laplacianineq}
        \int_\Om \abs{\Delta u}^p dx \geq 
        C_{CZ,p}\Big(\frac{p-1}{p} \Big)^p \int_\Om \frac{\abs{\nabla u}
        ^p}{\delta^p(x)} dx + C_{CZ,p}\Big(\frac{p-1}{p} \Big)^{p-1} \int_\Om \frac{\abs{\nabla u}
        ^p}{\delta^{p-1}(x)} (-\Delta\delta) dx.
    \end{equation}
    Particularly, if $-\Delta\delta \geq 0$ in $\Om\setminus\ridge$, then
     \begin{equation}   \label{ec:laplacianineq2}
        \int_\Om \abs{\Delta u}^p dx \geq C_{CZ,p}\Big(\frac{p-1}{p} \Big)^p \int_\Om \frac{\abs{\nabla u}
        ^p}{\delta^p(x)} dx.
    \end{equation}
Moreover, in view of \eqref{ec:laplacianineq2}, 
\begin{equation}    \label{constCZ}
    C_{p,N} \;\lambda_p \; \leq C_{CZ,p} \;\lambda_p \; \leq \; \lambda^\sharp \left(p,\Om,Z=\frac{1}{\delta^p}\right) \;\leq\; \lambda_p \;\;\;\;\; \text{and} \;\;\;\;\;  C_{p,N}\;\leq\; C_{CZ,p}\;\leq\; 1.
\end{equation}
\end{corollary}

\begin{remark}  \label{ec:remark-convexity}
    We note that the positivity of the last term of \eqref{ec:hessianineq} and \eqref{ec:laplacianineq} depends on the sign of $\Delta\delta$ which is related to the geometry of $\prt\Om$. From \cite{balinsky2} we know that $\Delta \delta$ depends on the principal curvatures of $\prt\Om$. It is known that for $\Om$ a convex domain it holds
    \begin{equation}   \label{ec:delta-1}
        -\Delta\delta(x) \geq 0, \;\;\forall x\in\Om\setminus\ridge
    \end{equation}
    and 
    \begin{equation}   \label{ec:delta-2}
        -\Delta\delta(x) \leq 0, \;\;\forall x\in\overline{\Om}^c.
    \end{equation}
\end{remark}
Note that there exist non-convex domains which satisfy conditions \eqref{ec:delta-1} (e.g., a torus with outer radius at least twice the inner radius). As we mentioned in the introduction, in \cite{lewis} is proved that condition \eqref{ec:delta-1} is equivalent to the weakly mean convexity of the domain. To be more precise, by definition, a domain with  $C^2$ boundary in $\R^N$ is \textit{weakly mean convex} if $H(y)\geq 0$ for any $y\in \prt\Om$, where $H(y)$ is the mean curvature at $y\in \prt\Om$. Recall that $H(y)=\frac{1}{N-1}\sum_{i=1}^{N-1}\kappa_i(y)$, where $\kappa_i(y)$'s are the principal curvatures of $\prt\Om$ at $y$. (see Section 3 below) Apart from the torus and different small perturbations of it or a torus with high genus, we note that an interesting example of weakly mean convex domains are domains with embedded minimal surfaces as boundaries, which we know have $H\equiv 0$ everywhere. These also qualify as examples for domains which satisfy \eqref{ec:delta-2}.\\
In the same paper \cite{lewis} it is proved the following statement: if $\Om\subset \R^N$ is a weakly mean convex domain and $H_0:=\inf_{y\in\prt\Om} H(y)$, then for any $1<p<\infty$ and $x\in\Om\setminus\ridge$ we have
\begin{equation}   \label{ec:mean-curv-bound}
    -\Delta\delta(x) \;\geq\; \frac{p H^p(y)}{N^{p-1}}\left(\frac{p}{p-1}\right)^{p-1} \delta^{p-1}(x) \;\geq\; \frac{p H_0^p}{N^{p-1}}\left(\frac{p}{p-1}\right)^{p-1} \delta^{p-1}(x),
\end{equation}
where $y\in\prt\Om$ is the near point of $x$ (see Section 3). As a direct consequence of this fact and inequalities \eqref{ec:hessianineq}, \eqref{ec:laplacianineq-p=2} and \eqref{ec:laplacianineq}, we have the next corollary.
\begin{corollary}   \label{}
    Let $N\geq 2$, $1<p<\infty$ and $\Om\subset\R^N$ weakly mean convex bounded domain with $\prt\Om$ belonging to class $C^2$. Then for any $ u\in C^\infty_c(\Om\setminus\ridge)$:
    \begin{equation}   \label{ec:hessianineq-mean}
        \int_\Om \abs{\hess u}^p dx \geq \Big(\frac{p-1}{p} \Big)^p \int_\Om \frac{\abs{\nabla u}
        ^p}{\delta^p(x)} dx +  \frac{p H_0^p}{N^{p-1}} \int_\Om \abs{\nabla u}
        ^p dx.
    \end{equation}
    If $p=2$
    \begin{equation}    \label{ec:laplacianineq-p=2-mean}
        \int_\Om \abs{\Delta u}^2 dx \geq \frac{1}{4}\int_\Om \frac{\abs{\nabla u}
        ^2}{\delta^2(x)} dx + \frac{2H_0^2}{N} \int_\Om \abs{\nabla u}
        ^2 dx,
    \end{equation}
    while for $p\in(1,\infty)\setminus\{2\}$ 
    \begin{equation}   \label{ec:laplacianineq-mean}
        \int_\Om \abs{\Delta u}^p dx \geq 
        C_{CZ,p}\Big(\frac{p-1}{p} \Big)^p \int_\Om \frac{\abs{\nabla u}
        ^p}{\delta^p(x)} dx + C_{CZ,p}\frac{p H_0^p}{N^{p-1}} \int_\Om \abs{\nabla u}
        ^p dx.
    \end{equation}
\end{corollary}
The following result gives sufficient conditions for the inequality \eqref{ec:hessianineq} to hold for functions in $C^\infty_c(\Om)$.
\begin{theorem}   \label{thmExtension}
    Let $N\geq 2$, $1<p<\infty$ and $\Om\subset\R^N$ an open bounded domain with $\prt\Om$ belonging to class $C^2$. Assume that the ridge $\ridge$ set can be written as an intersection of a decreasing family of open sets $(\Om_\eps)_{\eps>0}$ with smooth boundaries such that:
    \begin{enumerate}
        \item[i)] $\ridge$ has zero Lebesgue measure; 
        \item[ii)] for sufficiently small $\eps$: $\nabla\delta\cdot \nu_\eps \geq 0$, \;\;$x\in\prt\Om_\eps$,\;\; $\nu_\eps$ is the interior unit normal to $\prt\Om_\eps$;
        \item[iii)] $\Delta\delta \in L^1(\Om\setminus\Om_\eps)$. 
    \end{enumerate}
    Then inequality \eqref{ec:hessianineq} holds for any $u\in C^\infty_c(\Om)$.
\end{theorem}

Some domains that fit conditions $i)$, $iii)$ of Theorem \ref{thmExtension} are open balls and ellipses. Next, we state a new result for functions depending on the distance to the boundary in any domain, not necessarily weakly mean convex domains.
\begin{theorem}   \label{thm3}
    Let $N\geq 2$, $1< p<\infty$ and $\Om\subset\R^N$ an open domain with $\prt\Om$ belonging to class $C^2$. Then for any $u\in C^\infty_c(\Om\setminus\ridge)$ depending on the distance to the boundary $\left(i.e.\; u(x)=g(\delta(x))\right)$ it holds the following
    \begin{equation}   \label{ec:laplacianineq3}
        \int_\Om \abs{\Delta u}^p dx \geq \Big(\frac{p-1}{p} \Big)^p \int_\Om \frac{\abs{\nabla u}^p}{\delta^p(x)} dx + p\Big(\frac{p-1}{p} \Big)^p\int_\Om \frac{\abs{\nabla u}^p}{\delta^{p-1}(x)} \Delta\delta dx.
    \end{equation}
    Particularly, if $\Delta\delta \geq 0$ in $\Om\setminus\ridge$, then
    \begin{equation}   \label{ec:laplacianineq4}
        \int_\Om \abs{\Delta u}^p dx \geq \Big(\frac{p-1}{p} \Big)^p \int_\Om \frac{\abs{\nabla u}^p}{\delta^p(x)} dx.
    \end{equation}
\end{theorem}
According to Remark \ref{ec:remark-convexity}, the exterior of convex sets (i.e., $\Om:=\R^N\setminus \overline{D}$, $D$ is convex) fits to \eqref{ec:laplacianineq4} in Theorem \ref{thm3}.  

In the latter, we give some applications to a nonlinear problem, which we introduce below.  

    For any $1<p<\infty$ the $p-$Bilaplacian operator is formally defined for functions $u\in C_c^\infty(\Omega)$ as
     \begin{equation}   \label{ec:p-Bilaplaciandef}
        \Delta^2_p u := \Delta(\abs{\Delta u}^{p-2}\Delta u).
    \end{equation}
In the following, let $N\geq 3$, $\Omega$ a bounded domain with smooth boundary, $\lmb\in\R$, $1<p$ and $1<q$. We study the existence of nontrivial solutions for problem 
    \[\left\{\begin{array}{cc}  \tag{$P_\lmb$}\label{ec:systemP}
    \Delta^2_p u + \lmb\;\dvg\left(\frac{\abs{\nabla u}^{p-2} \nabla u}{\delta^p}\right) = \abs{u}^{q-2}u,  &  x\in\Om, \\
    u=\frac{\prt u}{\prt \n}=0,  &  x \in\prt\Om. 
\end{array} \right.\] 
Here, $\n$ denotes the outward unit normal to $\prt\Om$. Similar problems with singular potential given in the form of $\frac{1}{\abs{x}^\alpha}$ were studied before in works such as the comprehensive collection \cite{peral-monograph} and others, see \cite{bhakta}, \cite{drissi}, \cite{wang} and references therein. Problem \eqref{ec:systemP} seems to be new in literature and the aforementioned inequalities of type \eqref{HR} are crucial for the existence of solutions.\\
The appropriate functional space where we study the existence of weak solutions of \eqref{ec:systemP} is the Sobolev space $W^{2,p}_0(\Om)$ (see \cite{davies-hinz}), endowed with the norm 
\begin{equation}
    \|u\|_{W^{2,p}_0(\Om)}:=\|\Delta u\|_{L^p(\Om)}.
\end{equation}
Note that, for $\lmb<\lmb^\sharp\left(p,\Om,Z=\frac{1}{\delta^p}\right)$, we could also use the equivalent norm given by 
\begin{equation}   \label{ec:equivnorms}
\|u\|:= \left(\|\Delta u\|_{L^p(\Om)}^p - \lmb\int_\Om \frac{\abs{\nabla u}^p}{\delta^p(x)} dx \right)^\frac{1}{p}.
\end{equation}
This is due to the Hardy-Rellich inequality \eqref{HR}:
\begin{equation}
    \left(1-\frac{\lmb^+}{\lmb^\sharp\left(p,\Om,Z=\frac{1}{\delta^p}\right)}\right)\|u\|_{\W}^p \leq \|u\|^p \leq \left(1+\frac{\lmb^-}{\lmb^\sharp\left(p,\Om,Z=\frac{1}{\delta^p}\right)}\right)\|u\|_{\W}^p,
\end{equation}
where $\lmb^+$ and $\lmb^-$ are the positive and negative part of $\lmb$, respectively. Denote by $p^{**}:=\frac{Np}{N-2p}$ the critical Sobolev exponent in the embedding (\cite[Thoerem 7.22]{gilbardtrud})
\begin{equation}    \label{embed}
    W^{2,p}_0(\Om)\hookrightarrow L^q(\Om),
\end{equation}
which is compact for:
\begin{enumerate}
    \item[i)] any $1< q<p^{**}$ when $1<p<\frac{N}{2}$;\\
    \item[ii)] any $1<q$ when $p\geq \frac{N}{2}$.
\end{enumerate}
\begin{definition}   \label{weak-sol}
We say that $u\in\W$ is a weak solution to the problem \eqref{ec:systemP} if it satisfies the following weak formulation:
\begin{equation}    \label{weak-sol}
\int_\Om \abs{\Delta u}^{p-2}\Delta u \Delta\phi dx - \lmb\int_\Om \frac{\abs{\nabla u}^{p-2}}{\delta^p(x)}\nabla u\nabla \phi dx = \int_\Om \abs{u}^{q-2}u\phi dx, \;\;\; \forall \phi\in\W.    
\end{equation}
\end{definition}
It is straightforward, by integration by parts, that any possible classical solution of \eqref{ec:systemP} is a weak solution.
\begin{theorem}   \label{thmExistence}
    There exists at least one non-trivial weak solution for the problem \eqref{ec:systemP} in the following cases:
    \begin{enumerate}
        \item[1.]\; for any $1<q<p$ and any $\lmb<\lmb^\sharp(p, \Om, Z=\frac{1}{\delta^p})$;\\
        \item [2.]\; for any $1<p<\frac{N}{2}$, $p<q<p^{**}$ and any $\lmb\leq 0$.
    \end{enumerate}
\end{theorem}
Moreover, we prove a non-existence result in a ball centered at $0$ (for simplicity, consider the ball of radius $1$).
\begin{theorem}   \label{thmNonExistence}
    If $\Om$ is the unit ball in $\R^N$, $N\geq 2$, $1<p<\frac{N}{2}$ and either ($\lmb<0$ and $q\geq p^{**}$), or ($\lmb\leq 0$ and $q> p^{**}$),    
    then there exist no non-trivial solutions of \eqref{ec:systemP}.
\end{theorem}

The rest of the paper is structured as follows.  Section 3 is dedicated to some preliminary geometric concepts and useful properties that we use later. In Sections 4 and 5 we prove the main theorems regarding the Hardy-Rellich inequalities, while in Sections 6 and 7 we give an application of our results in proving the existence and non-existence of non-trivial solutions for problem \eqref{ec:systemP} via variational methods (direct method and mountain pass theorems) and a Pohozaev-type identity, respectively.

\section{Properties of distance function to the boundary}
\subsection{Regularity}
Our purpose in this section is to describe the regularity of the distance function $\delta$. As we shall see, there is a strong connection between the regularity of $\delta$ and the geometry of $\Om$ and its boundary. These are known properties, treated in works such as \cite{evans}, \cite{fremlin}, \cite{federer} and in the Appendix of \cite{gilbardtrud}. For a brief introduction to this subject, we will follow the recent work of the authors in \cite{balinsky1} and cite the other appropriate references when needed.
Suppose that $\Om$ is a domain in $\R^N$, $N\geq 2$. If $x$ is a point in $\Om$, we denote by $N(x):=\{y\in\prt\Om \;|\; \delta(x)=\abs{x-y}\}$ the \textit{near set} of $x$ on $\prt\Om$. If $N(x)=\{y\}$, we write $y=N(x)$. For $x\in\Om$ and $y\in N(x)$ put $z:=tx+(1-t)y$, where $0<t<1$. Then $N(z)=y$. This tells us that for any point on the segment $[x,y]$ the only near point on the boundary is precisely $y$. Moreover, if we put 
$$\mu:=\sup\{t>0 \;|\; y\in N(y+t[x-y]) \},$$
then, for all $t\in(0,\mu)$, $N(y+t[x-y])=y.$ This translates to the fact that there is a precise point on the line determined by $x$ and $y$ with the property that any point on the line up to this one has $y$ as its near point.
This precise point is defined as $p(x):=y+\mu(x-y)$ and it's called the ridge point of $x$ in $\Om$. The set $\ridge:=\{p(x)\;|\;x\in\Om\}$ is called the \textit{ridge set} of $\Om$. \\
\begin{remark}[\cite{balinsky1}, p. 51]
    Note that if $\Om$ doesn't contain a half-space, the definition of $p(x)$ is consistent; otherwise, for some $x\in\Om$ we have to put $p(x)=\infty$.  It follows that if $x\notin\ridge$ then $\card N(x) =1$, i.e., $N(x)$ is a unique point. The converse is not always true (e.g., $\Om$ an ellipse).
\end{remark}
This remark leads us to the next definition. The \textit{skeleton} of $\Om$ is the set 
\begin{equation}
    \skelet := \{x\in\Om \;|\;\card N(x)>1 \}.
\end{equation}
\begin{prop} [\cite{balinsky1}, p. 51]
    The function $\delta$ is differentiable at $x\in\Om$ if and only if $\card N(x)=1$. If $N(x)=y$, then $\nabla\delta(x)=\frac{x-y}{\abs{x-y}}$ is continuous on its domain of definition.
\end{prop}
It is clear that $\skelet$ is exactly the set of points in $\Om$ at which $\delta$ is not differentiable. It can be shown that $\delta$ is Lipschitz continuous; hence, by Rademacher's Theorem, $\skelet$ is of zero Lebesgue measure. The following relation takes place:
$$\skelet\subseteq \ridge\subseteq \overline{\skelet}.$$
When $\Om$ is bounded, $\ridge\neq\emptyset$. When $\Om$ has a convex complement set $\overline{\Om}^c$ then $\skelet=\ridge=\emptyset$. It is an open question if $\ridge$ has zero Lebesgue measure in general. It is known that if $\Om$ is a proper open subset of $\R^2$, then the ridge has zero Lebesgue measure. It is not a closed set in general, since in \cite[pg. 10]{mantegazza} the authors propose an example of a convex domain in $\R^2$ with $C^{1,1}$-boundary for which $\overline{\skelet}$ has non-zero Lebesgue measure, so $\ridge$ is not closed in this case. Also, if $\Om$ is a bounded domain with $C^{2,1}$-boundary (see, e.g., \cite[Definition 1.3.7]{balinsky1}), then $\ridge$ has finite $(N-1)$-dimensional Hausdorff measure, hence implying that it has zero Lebesgue measure.  \\
For simplicity and clarity of exposition, we will assume in the rest of the paper that the domains under consideration satisfy $\overline\skelet=\overline{\ridge}$, which covers the most important cases of domains.\\
According to \cite[Section 2.4]{balinsky1}, if $\Om$ has a $C^2$-boundary, then $\delta\in C^2(\Om\setminus\ridge)$.

\subsection{First and second order derivatives in the smooth case}
Assume that $\prt\Om\in C^2$, which means that $\prt\Om$ is the graph of a $C^2$-function. Let $(\gamma, U)$ be a local parametrization of $\partial \Omega$ with $\gamma: U\subset \R^{N-1}\rightarrow \R^N$ such that the curves $s_i\mapsto \gamma(s_1, \ldots, s_i, \ldots s_{N-1})=\gamma (s')$ are parametrized by the arclength, i.e. 
$\|v_i:=\frac{\prt\gamma}{\prt s_i}(s')\|=1$ for any $i=1,..., N-1$, and have the following properties
\begin{enumerate}
    \item[i)] $<v_i,v_j> = \delta_{ij}$\;\; \text{and}\;\; $<\frac{\prt v_i}{\prt s_j}, v_j> =0$,\;\; $\forall i,j=1,..,N-1$,  
    \item[ii)] $<v_i,\nu> = 0$,\;\; $\forall i=1,..,N-1$,
    \item[iii)]  $\frac{\prt \nu}{\prt s_i}(s') = \kappa_i(s') \frac{\prt \gamma}{\prt s_i}(s') = \kappa_i(s')v_i(s')$,\;\; $\forall i=1,..,N-1$,
\end{enumerate}
where $\nu(s')$ is the interior normal vector at $\gamma(s')$, the coefficients $\kappa_i(s')$ are the principal curvatures of $\prt\Om$ at the point $\gamma(s')$, which are scalars given by the proportionality of the colinear vectors $v_i$ and $\frac{\prt \nu}{\prt s_i}$. \\
Consider 
\begin{equation}
    \Sigma:=\{s=(s', s_N)\in \R^N \ |\ s' \in U, \  0\leq s_N\leq \dist(\gamma(s'), \ridge)\}
\end{equation}
and the change of variables $\Gamma: \Sigma\rightarrow \Om\setminus\ridge$,
\begin{equation}
    \Gamma(s)=\gamma(s_1,...,s_{N-1})+s_N \nu(s_1,...,s_{N-1})=:x,
\end{equation}
where $\gamma(s')\in\prt\Om$ is the unique near point of $x$ on the boundary of $\Om$, $\nu(s')$ is the interior normal vector at $\gamma(s')$ pointing in the direction of $x$ and $s_N=\delta(x)$. By \cite{balinsky2}, we know that 
\begin{equation}
    \delta_{x_i}(x) = \nu^i.   \label{ec:derivate1}
\end{equation}
We differentiate again with respect to $x_j$ and we get
\begin{align}     \label{ec:deriv-delta}
    \delta_{x_ix_j}(x) \;&=\; \sum_{l=1}^N \frac{\prt \nu^i}{\prt s_l} \frac{\prt s_l}{\prt x_j}    \;=\; \sum_{l=1}^{N-1} \frac{1}{1+\delta \kappa_l} v_l^i \frac{\prt \nu^j}{\prt s_l}  + \nu^i\frac{\prt\nu^j}{\prt\delta}     \notag\\
    & =\; \sum_{l=1}^{N-1} \frac{\kappa_l}{1+\delta \kappa_l} v_l^i v_l^j.
\end{align}
We remark that, since $v_l$ has length $1$, $\delta_{x_ix_j}(x)$ is bounded when $x$ is not close to the ridge $\ridge$, i.e., $1+\delta\kappa_l >> 0$ and there exists $C_V>0$ such that
\begin{equation}    \label{deriv-delta-bound}
    \|\delta_{x_ix_j}\|_{L^\infty(\Omega\setminus V)} < C_V
\end{equation}
for any fixed neighborhood $V$ of $\ridge$, small enough to be contained in $\Omega$. This estimate is of great importance to the proof of the main theorem in the next section. We recall the following result from \cite[Section 2.4]{balinsky1}: if $\Om$ has a $C^2$-boundary, then for $g\in C^2(\R_+)$, $g(x)=g(\delta(x))$, the next formula holds:
\begin{equation}     \label{ec:s-laplacian}
    \Delta_x g(x)=\frac{\prt^2 g}{\prt\delta^2}+ \sum_{i=1}^{N-1}\frac{\kappa_i}{1+\delta\kappa_i} \frac{\prt g}{\prt\delta}.
\end{equation}

\section{Second order inequalities with the Hessian}
\begin{proof}[\textbf{Proof of Theorem \ref{thm1}}] We first give a proof for the inequality \eqref{ec:hessianineq}. By the Eikonal equation $|\nabla\delta|=1$ we have the following identity
\begin{align}
    \dvg(\nabla \delta \delta^{1-p}) & = \Delta\delta\delta^{1-p} + (1-p)\delta^{-p}\nabla\delta\cdot\nabla\delta   \notag\\
    & = \Delta\delta\delta^{1-p} + (1-p)\delta^{-p}, \quad \mathrm{x\in \Omega \setminus\ridge}.   \notag
\end{align}
Then we can write
\begin{equation}   \label{ec:dvgdelta}
    \delta^{-p} = \frac{1}{1-p} \Big[ \dvg(\nabla\delta\delta^{1-p})-\Delta\delta\delta^{1-p} \Big].
\end{equation}
Hence, for $u\in C^\infty_c(\Om\setminus\ridge)$, by \eqref{ec:dvgdelta}, integration by parts and Cauchy-Schwarz inequality we obtain 
\begin{align}
    \int_\Om \frac{\abs{\nabla u}^p}{\delta^p} & = \frac{1}{1-p} \int_\Om \Big[ \dvg(\delta^{1-p} \nabla\delta)-\Delta\delta\delta^{1-p} \Big] \abs{\nabla u}^p dx  \notag\\
    & = \frac{1}{p-1} \int_\Om \delta^{1-p} \nabla\delta\nabla(\abs{\nabla u}^p) dx + \frac{1}{p-1} \int_\Om \frac{\abs{\nabla u}^p}{\delta^{p-1}(x)} \Delta\delta dx   \notag\\
    & = \frac{p}{p-1} \int_\Om \delta^{1-p} \abs{\nabla u}^{p-2} \nabla\delta \;\hess u \;(\nabla u)^T \; dx + \frac{1}{p-1} \int_\Om \frac{\abs{\nabla u}^p}{\delta^{p-1}(x)} \Delta\delta dx  \notag\\
    & \leq \frac{p}{p-1} \int_\Om \delta^{1-p} \abs{\nabla u}^{p-1} |\;\hess u| \; dx + \frac{1}{p-1} \int_\Om \frac{\abs{\nabla u}^p}{\delta^{p-1}(x)} \Delta\delta dx.  \notag
\end{align}
By applying the Young Inequality $ab\leq \frac{1}{p}a^p+\frac{p-1}{p}b^\frac{p}{p-1}$ with $a=\frac{p}{p-1}\hess u$ and $b=\delta^{1-p}\abs{\nabla u}^{p-1}$ we get
\begin{align}
    \int_\Om \frac{\abs{\nabla u}^p}{\delta^p}& \leq \frac{1}{p}\left(\frac{p}{p-1}\right)^p \int_\Om \abs{\hess u}^p dx + \frac{p-1}{p}\int_\Om \left( \frac{\abs{\nabla u}^{p-1}}{\delta^{p-1}(x)}\abs{\nabla\delta} \right)^\frac{p}{p-1} dx + \frac{1}{p-1} \int_\Om \frac{\abs{\nabla u}^p}{\delta^{p-1}(x)} \Delta\delta dx    \notag\\
    & = \frac{1}{p}\left(\frac{p}{p-1}\right)^p \int_\Om \abs{\hess u}^p dx + \frac{p-1}{p} \int_\Om \frac{\abs{\nabla u}^p}{\delta^p} dx + \frac{1}{p-1} \int_\Om \frac{\abs{\nabla u}^p}{\delta^{p-1}(x)} \Delta\delta dx.    \notag
\end{align}
Rearranging the terms, we obtain:
\begin{equation}   
        \int_\Om \abs{\hess u}^p dx \geq \left(\frac{p-1}{p} \right)^p \int_\Om \frac{\abs{\nabla u}
        ^p}{\delta^p(x)} dx + \left(\frac{p-1}{p} \right)^{p-1} \int_\Om \frac{\abs{\nabla u}
        ^p}{\delta^{p-1}(x)} (-\Delta\delta) dx,    \notag
\end{equation}
which finishes the proof of \eqref{ec:hessianineq}.  \\
If we assume that $-\Delta\delta\geq 0$ in $\Om\setminus\ridge$, then we restrict the inequality to the following form
\begin{equation}
    \int_\Om \abs{\hess u}^p dx \geq \left(\frac{p-1}{p} \right)^p \int_\Om \frac{\abs{\nabla u}^p}{\delta^p(x)} dx.   \notag
\end{equation}
In order to prove the optimality of the constant in \eqref{ec:hessianineq2} we need to find a sequence $(u_\eps)_\eps$ in $C^\infty_c(\Om\setminus\ridge)$ such that 
\begin{equation}
    \frac{\int_\Om \abs{\hess u_\eps}^p dx}{\int_\Om \frac{\abs{\nabla u_\eps}^p}{\delta^p(x)} dx}  \;\;\stackrel{\eps\to 0}{\longrightarrow}\;\; \Big(\frac{p-1}{p} \Big)^p.    \label{quotient}
\end{equation}
We have to design a minimizing sequence that vanishes as we approach the ridge set and also when we approach the boundary, so in $C^\infty_c(\Om\setminus\ridge)$. As we shall see further, the behavior of this sequence near the boundary is extremely important for the sharpness of the constant. \\

Let $r,\eps>0$ sufficiently small, say $r<\frac{\dist(\ridge,\prt\Om)}{8}$ and $\eps<\min\left\{1,\frac{\dist(\ridge,\prt\Om)}{8}\right\}$. Define the set $V_{r}:=\big\{ x\in\Om \;|\; \dist(x,\ridge)<r \big\} $. By the choice of $r$, $V_r$ and $V_{2r}$ are both subsets of $\Om$. Consider $\theta\in C^\infty(\R)$ such that 
\begin{equation}
\theta(\xi) = \begin{cases}
    0, \;&\;\; \text{if}\; \xi<\frac{1}{4};  \\
    1, \;&\;\; \text{if}\; \xi>\frac{1}{2}.   \notag
\end{cases}
\end{equation}
Inspired by \cite{cassano2}, we consider the following smooth cut-off function, $0\leq \theta_\eps\leq 1$:
\begin{equation}
\theta_\eps(x) = \begin{cases}
    0, \;&\;\; \text{if}\;\; \delta(x)<\eps^2 \;\;\text{or}\;\; x\in V_r,  \\
    \theta\left(\frac{\ln\frac{\delta(x)}{\eps^2}}{\ln\frac{1}{\eps}}\right), \;&\;\; \text{if}\;  \eps^2\leq \delta(x)\leq\eps,   \\
    1, \;&\;\; \text{if}\;\; \delta(x)>\eps \;\;\text{and}\;\; x\in \Om\setminus V_{2r}.
\end{cases}
\end{equation}
such that $\theta_\eps$ does not depend on $\eps$ on $V_{2r}\setminus V_r$. Denote by $x_\eps:=\frac{\ln\frac{\delta(x)}{\eps^2}}{\ln\frac{1}{\eps}}$. We define, for any $x\in\Om\setminus\ridge$,
$$u_\eps(x)= \delta^{\frac{2p-1}{p}+\eps}(x)\theta_\eps(x).$$
Note that $u_\eps\in C^\infty_c(\Om\setminus V_r) \subset C^\infty_c(\Om\setminus\ridge)$. We compute its derivatives for $x\in\Om\setminus\ridge$:
\begin{align}   
    u_{\eps,x_i}(x) & = \left(\frac{2p-1}{p}+\eps\right)\delta^{\frac{p-1}{p}+\eps}(x)\delta_{x_i}(x)\theta_\eps(x) + \delta^{\frac{2p-1}{p}+\eps}(x)\theta_{\eps,x_i}(x),   \notag\\
    u_{\eps,x_ix_j}(x) & = \left(\frac{2p-1}{p}+\eps\right)\left(\frac{p-1}{p}+\eps\right)\delta^{-\frac{1}{p}+\eps}(x) \delta_{x_i}(x)\delta_{x_j}(x)\theta_\eps(x)   \notag\\
    &+ \left(\frac{2p-1}{p}+\eps\right)\delta^{\frac{p-1}{p}+\eps}(x)\delta_{x_ix_j}(x)\theta_\eps(x) + \left(\frac{2p-1}{p}+\eps\right) \delta^{\frac{p-1}{p}+\eps}(x)\delta_{x_i}(x)\theta_{\eps,x_j}(x)  \notag\\
     & + \left(\frac{2p-1}{p}+\eps\right)\delta^{\frac{p-1}{p}+\eps}(x)\delta_{x_j}(x)\theta_{\eps,x_i}(x) + \delta^{\frac{2p-1}{p}+\eps}(x)\theta_{\eps,x_ix_j}(x).   \notag
\end{align}
We need to investigate the behavior of the derivatives of $\theta_\eps$ as $\eps\to 0$. From now on, we shall use the symbols ”$\lesssim$” and "$\gtrsim$", meaning that the inequality we state holds up to some universal constants, usually denoted by $C$, depending on the $L^\infty-$norm of $\theta$, $\delta$ and their derivatives in the respective region. First, for $x\in V_{2r}\setminus V_r$, since $\theta$ is independent of $\eps$, there exist constants $C_r$, $C_r'$ and $C_r''$ (i.e., depending on $r$) such that
\begin{align}     
    \|&\delta(x)\|_{L^\infty(V_{2r}\setminus V_r)} \;,\;\; \left\|\frac{1}{\delta(x)}\right\|_{L^\infty(V_{2r}\setminus V_r)} \leq C_r    \label{delta-bound} \\
    \|&D\delta\|_{L^\infty(V_{2r}\setminus V_r)}\;,\;\; \|D^2\delta\|_{L^\infty(V_{2r}\setminus V_r)}\;\; \leq \; C_r'   \label{deriv-delta-bound} \\
    \|&\theta_\eps\|_{L^\infty(V_{2r}\setminus V_r)}\;,\;\; 
    \|D\theta_\eps\|_{L^\infty(V_{2r}\setminus V_r)}\;,\;\; \|D^2\theta_\eps\|_{L^\infty(V_{2r}\setminus V_r)}\;\; \leq \; C_r''.   \label{deriv-theta-bound}
\end{align}
On the other hand, on $\{x\in\Omega\;|\;\eps^2\leq \delta(x)\leq\eps\} \stackrel{not}{=} \{\eps^2\leq \delta(x)\leq\eps\}$, we have 
\begin{align}    
    \theta_{\eps,x_i}(x) & = \frac{1}{\ln{(\frac{1}{\eps})}} \frac{\theta_{x_i}(x_\eps)\delta_{x_i}(x)}{\delta(x)},    \notag\\
    \theta_{\eps,x_ix_j}(x)& = \frac{1}{\ln{(\frac{1}{\eps})}} \Bigg[ \frac{1}{\ln{(\frac{1}{\eps})}}\frac{\theta_{x_ix_j}(x_\eps)\delta_{x_i}(x)\delta_{x_j}(x)}{\delta(x)} 
    + \frac{\theta_{x_i}(x_\eps)\delta_{x_ix_j}(x)}{\delta(x)} 
    - \frac{\theta_{x_i}(x_\eps)\delta_{x_i}(x)\delta_{x_j}(x)}{\delta^2(x)}  \Bigg].      \notag
\end{align}
Consequently, we have the following estimates in $\eps$:
\begin{align}
    \left|\theta_{\eps,x_i} \right| & \lesssim \frac{\|D\theta\|_{L^\infty(\Om)}}{\delta(x)}\;  \frac{1}{\ln{(\frac{1}{\eps})}}     \label{deriv-1-thetaeps}\\
    \left|\theta_{\eps,x_ix_j} \right| & \lesssim \Bigg[ \frac{1}{\ln^2{(\frac{1}{\eps})}}  \frac{\|D^2\theta\|_{L^\infty(\Om)}}{\delta(x)}  
    + \frac{1}{\ln{(\frac{1}{\eps})}} \frac{\|D\theta\|_{L^\infty(\Om)} \|\delta_{x_ix_j}\|_{L^\infty(\Om\setminus V_r)}}{\delta(x)} 
    + \frac{1}{\ln{(\frac{1}{\eps})}} \frac{\|D\theta\|_{L^\infty(\Om)}}{\delta^2(x)} \Bigg].    \label{deriv-2-thetaeps}
\end{align}

Let us now compute the integrals in \eqref{quotient}. The integral in the denominator is
\begin{align}
    \int_\Om \frac{\abs{\nabla u_\eps}^p}{\delta^p(x)} dx 
    & = \int_\Om \frac{\abs{ \left(\frac{2p-1}{p}+\eps\right)\delta^{\frac{p-1}{p}+\eps}(x)\theta_\eps(x) \nabla\delta  + \delta^{\frac{2p-1}{p}+\eps}(x)\nabla\theta_\eps }^p}{\delta^p(x)} dx    \notag\\
    & = \left(\frac{2p-1}{p}+\eps\right)^p \int_\Om \delta^{-1+p\eps}(x)\theta_\eps^p(x) dx + I_{1,\eps},  \label{ec:i1}
\end{align}
where
\begin{equation}
    I_{1,\eps} =\int_\Om \abs{ \left(\frac{2p-1}{p}+\eps\right)\delta^{-\frac{1}{p}+\eps}\theta_\eps \nabla\delta  + \delta^{\frac{p-1}{p}+\eps}\nabla\theta_\eps }^p dx -  \left(\frac{2p-1}{p}+\eps\right)^p \int_\Om \delta^{-1+p\eps}\theta_\eps^p dx.   \notag
\end{equation}
The integral in the numerator is
\begin{align}
    \int_\Om \abs{\hess u_\eps}^p dx 
    & = \int_\Om \Big( \sum_{i,j} u_{\eps,x_ix_j}^2(x) \Big)^\frac{p}{2} dx    \notag\\
    & = \left(\frac{2p-1}{p} + \eps\right)^p \left(\frac{p-1}{p} + \eps \right)^p \int_\Om \delta^{-1+p\eps}(x) \theta_\eps^p(x)  dx + I_{2,\eps}  \label{ec:i2}
\end{align}
where 
\begin{equation}
    I_{2,\eps} = \int_\Om  \Big( \sum_{i,j} u_{\eps,x_ix_j}^2(x) \Big)^\frac{p}{2} dx - \left(\frac{2p-1}{p} + \eps\right)^p \left(\frac{p-1}{p} + \eps\right)^p \int_\Om \delta^{-1+p\eps}(x) \theta_\eps^p(x)  dx,  \notag
\end{equation}
First, let us justify that the first terms in \eqref{ec:i1} and \eqref{ec:i2} tend to $\infty$ as $\eps\to 0$. In the following, we denote 
$$M_1(r)=\min_{x\in\prt V_{r}} \delta(x) \;\;\text{and}\;\; M_2(r)=\min_{x\in\prt V_{2r}} \delta(x).$$
Recall that $\theta_\eps=0$ on $V_r$ and on $\{x\in\Om\;|\;\delta(x)<\eps^2\}\stackrel{not}{=}\{\delta(x)<\eps^2\}$. Using this, we obtain 
\begin{align}
    \int_\Om \delta^{-1+p\eps}(x)\theta_\eps^p(x) \;dx 
    & \;\;=\; \int_{(\Om\setminus V_r)\cap\{\delta(x) >\eps^2\}} \delta^{-1+p\eps}(x) \theta_\eps^p(x) \;dx \;\geq\;  \int_{(\Om\setminus V_{2r})\cap\{\delta(x) >\eps\}} \delta^{-1+p\eps}(x) \;dx.  \label{int-infinity}
\end{align}
Let $\Sigma_s=\{x\in\Om \;|\; \delta(x)=s \}$. If $|\Sigma_s|$ denotes the $(N-1)-$dimensional Hausdorff measure of $\Sigma_s$ then the function $s\to|\Sigma_s|$ is continuous, bounded and there exists $C_\Sigma>0$ such that $|\Sigma_s|\geq C_\Sigma>0$ for any $\eps<s<M_2(r)$. Using the co-area formula, we get
\begin{align}
    \int_{(\Om\setminus V_{2r})\cap\{\delta(x) >\eps\}} \delta^{-1+p\eps}(x) \;dx 
& \;\geq\; \int_\eps^{M_2(r)} \int_{\Sigma_s} s^{-1+p\eps} d\sigma ds  \;=\; \int_\eps^{M_2(r)} s^{-1+p\eps} \abs{\Sigma_s} ds   \notag\\
    & \gtrsim\; \int_\eps^{M_2(r)} s^{-1+p\eps} ds  \;=\;  \frac{(M_2(r))^{p\eps}-\eps^{p\eps}}{p\eps} \;\; \stackrel{ \eps\to 0}{\longrightarrow} \infty.    \label{ec:convergence1}
\end{align} 
Thus, the first terms in \eqref{ec:i1} and \eqref{ec:i2} tend to infinity. We now estimate $I_{1,\eps}$ and $I_{2,\eps}$ and show that they are bounded as $\eps$ approaches $0$. Recall that
\begin{align}
    I_{1,\eps} = \int_\Om \abs{ \left(\frac{2p-1}{p} + \eps\right)\delta^{\frac{-1}{p}+\eps}(x)\theta_\eps(x) \nabla\delta  + \delta^{\frac{p-1}{p}+\eps}(x)\nabla\theta_\eps}^p - \left(\frac{2p-1}{p}+\eps\right)^p \delta^{-1+p\eps}(x)\theta_\eps^p(x) dx.   \notag
\end{align}
Notice that in $I_{1,\eps}$ we integrate only on the support of $\theta_\eps$, hence 
$I_{1,\eps}$ can be divided in three integrals, on three distinct regions, as follows:
\begin{align}
    T_{1,\eps} &:= \int_{V_{2r}\setminus V_r} \abs{ \left(\frac{2p-1}{p} + \eps\right)\delta^{\frac{-1}{p}+\eps}(x)\theta_\eps(x) \nabla\delta  + \delta^{\frac{p-1}{p}+\eps}(x)\nabla\theta_\eps}^p   \notag\\
    & \quad\quad\quad\quad\quad\quad\quad\quad\quad\quad\quad\quad\quad\quad\quad\quad\quad\quad\quad\quad- \left(\frac{2p-1}{p}+\eps\right)^p \delta^{-1+p\eps}(x)\theta_\eps^p(x) dx      \notag\\
    T_{2,\eps} &:= \int_{\{\delta(x)>\eps\}\setminus V_{2r}} \abs{ \left(\frac{2p-1}{p} + \eps\right)\delta^{\frac{-1}{p}+\eps}(x)\theta_\eps(x) \nabla\delta  + \delta^{\frac{p-1}{p}+\eps}(x)\nabla\theta_\eps}^p   \notag\\
    & \quad\quad\quad\quad\quad\quad\quad\quad\quad\quad\quad\quad\quad\quad\quad\quad\quad\quad\quad\quad- \left(\frac{2p-1}{p}+\eps\right)^p \delta^{-1+p\eps}(x)\theta_\eps^p(x) dx     \notag\\
    T_{3,\eps} &:=  \int_{\{\eps^2\leq \delta(x)\leq \eps\}} \abs{ \left(\frac{2p-1}{p} + \eps\right)\delta^{\frac{-1}{p}+\eps}(x)\theta_\eps(x) \nabla\delta  + \delta^{\frac{p-1}{p}+\eps}(x)\nabla\theta_\eps}^p   \notag\\
    & \quad\quad\quad\quad\quad\quad\quad\quad\quad\quad\quad\quad\quad\quad\quad\quad\quad\quad\quad\quad- \left(\frac{2p-1}{p}+\eps\right)^p \delta^{-1+p\eps}(x)\theta_\eps^p(x) dx.    \notag
\end{align}
Taking into account \eqref{delta-bound}, \eqref{deriv-delta-bound} and \eqref{deriv-theta-bound}, we may easily conclude that $T_{1,\eps}$ is uniformly bounded when $\eps\to 0$. \\
Since $\theta_\eps(x)=1$ for $x\in \{\delta(x)>\eps\}\setminus V_{2r}$, then $\nabla\theta_\eps$ vanishes on this set, so $T_{2,\eps}=0$. For $x\in\{\eps^2\leq \delta(x)\leq\eps\}$:
\begin{equation}     \label{deriv-theta-eps}
\theta_\eps(x)\leq 1 \;\;\text{and}\;\; \nabla\theta_\eps=\frac{1}{\ln\left(\frac{1}{\eps}\right)} \frac{\nabla\theta\nabla\delta}{\delta}.
\end{equation}
For $p\geq 2$, using \eqref{ec:ineqshafrir1} for
\begin{equation}
    a= \Big(\frac{2p-1}{p}\Big)\delta^{\frac{-1}{p}+\eps}(x)\theta_\eps(x) \nabla\delta  \;\;\text{and}\;\; b= \delta^{\frac{p-1}{p}+\eps}(x)\nabla\theta_\eps,   \notag
\end{equation}
and applying Cauchy-Schwarz inequality and \eqref{ineqshafrir4}, we get 
\begin{align}
    T_{3,\eps} &\lesssim \int_{\{\eps^2\leq \delta(x)\leq \eps\}}\Big(\abs{a}+\abs{b}\Big)^{p-2} \abs{b}^2 dx + \int_{\{\eps^2\leq \delta(x)\leq \eps\}}\abs{a}^{p-2}a\cdot b \;dx  \notag\\
    & \lesssim \int_{\{\eps^2\leq \delta(x)\leq \eps\}}\big(\abs{a}^{p-2}+\abs{b}^{p-2}\big) \abs{b}^2 dx + \int_{\{\eps^2\leq \delta(x)\leq \eps\}}\abs{a}^{p-1} \abs{b} dx   \notag\\
    & \lesssim  \int_{\{\eps^2\leq \delta(x)\leq \eps\}}\abs{a}^{p-2}\abs{b}^2 dx + \int_{\{\eps^2\leq \delta(x)\leq \eps\}}\abs{a}^{p-1}\abs{b} dx + \int_{\{\eps^2\leq \delta(x)\leq \eps\}}\abs{b}^p dx.    \notag
\end{align}
By \eqref{deriv-1-thetaeps} and \eqref{deriv-theta-eps},
\begin{align}    \label{T-3eps}
    T_{3,\eps} & \lesssim \int_{\{\eps^2\leq \delta(x)\leq \eps\}} \delta^{1+p\eps}(x)|\nabla\theta_\eps|^2 dx + \int_{\{\eps^2\leq \delta(x)\leq \eps\}} \delta^{p\eps}(x)|\nabla\theta_\eps| dx + \int_{\{\eps^2\leq \delta(x)\leq \eps\}} \delta^{p-1+p\eps}(x)\abs{\nabla\theta_\eps}^p dx    \notag\\
    & \lesssim \int_{\{\eps^2\leq \delta(x)\leq \eps\}}\delta^{-1+p\eps}(x) \left[ \frac{1}{\ln^2\left(\frac{1}{\eps}\right)} + \frac{1}{\ln\left(\frac{1}{\eps}\right)} + \frac{1}{\left|\ln\left(\frac{1}{\eps}\right) \right|^p} \right] dx    \notag\\
    & \lesssim \frac{1}{\ln\left(\frac{1}{\eps}\right)}\int_{\{\eps^2\leq \delta(x)\leq \eps\}}\delta^{-1+p\eps}(x) dx.   
\end{align}
By the co-area formula, we get
\begin{align}
    T_{3,\eps} &\lesssim \frac{1}{\ln\left(\frac{1}{\eps}\right)} \int_{\eps^2}^\eps s^{-1+p\eps} |\Sigma_s| ds \lesssim \frac{1}{\ln\left(\frac{1}{\eps}\right)} \frac{\eps^{p\eps}-\eps^{2p\eps}}{p\eps}.
\end{align}
Since 
\begin{equation}
    \frac{\eps^{p\eps}-\eps^{2p\eps}}{p\eps} \simeq -\ln\eps, \;\;\;\; \text{when}\;\;\eps\to 0,    \label{limita-eps}
\end{equation}
we conclude that $T_{3,\eps}$ is bounded for $p\geq 2$.
In the case when $1<p<2$, using the same notation for $a$ and $b$, one uses \eqref{ec:ineqshafrir2} and \eqref{deriv-theta-eps} to get the estimates:
\begin{align}
     T_{3,\eps} &\lesssim \int_{\{\eps^2\leq \delta(x)\leq \eps\}} \abs{a}^{p-2}a\cdot b dx + \int_{\{\eps^2\leq \delta(x)\leq \eps\}} \abs{b}^p dx  \notag\\
    & \lesssim \int_{\{\eps^2\leq \delta(x)\leq \eps\}} \abs{a}^{p-1} \abs{b} dx + \int_{\{\eps^2\leq \delta(x)\leq \eps\}} \abs{b}^p dx \notag\\
    & = \int_{\{\eps^2\leq \delta(x)\leq \eps\}} \delta^{p\eps}(x) \abs{\nabla\theta_\eps} dx + \int_{\{\eps^2\leq \delta(x)\leq \eps\}} \delta^{p-1+p\eps}(x)\abs{\nabla\theta_\eps}^p dx.    \notag
\end{align}
These are again bounded, as seen before in \eqref{T-3eps}. We conclude that $I_{1,\eps}$ is bounded. \\
We now turn our attention to $I_{2,\eps}$. Recall that
\begin{equation}
    I_{2,\eps} = \int_\Om  \Big( \sum_{i,j} u_{\eps,x_ix_j}^2(x) \Big)^\frac{p}{2} - \left(\frac{2p-1}{p} + \eps\right)^p \left(\frac{p-1}{p} + \eps\right)^p \delta^{-1+p\eps}(x) \theta_\eps^p(x)  dx.   \notag
\end{equation}
First, let us start with some estimates on $u_{\eps,x_ix_j}^2$ in $\Om\setminus\ridge$:
\begin{align}
    &u_{\eps,x_ix_j}^2(x) \leq \left( \frac{2p-1}{p}+\eps \right)^2\left( \frac{p-1}{p}+\eps \right)^2 \theta_\eps^2(x)\delta^{-\frac{2}{p}+2\eps}(x)\delta_{x_i}^2(x)\delta_{x_j}^2(x) + \delta^{\frac{4p-2}{p}+2\eps}(x)\theta_{\eps,x_ix_j}^2(x) \notag\\
    &+ C\Bigg( \delta^{\frac{2p-2}{p}+2\eps}(x) \Big[ \theta_\eps^2(x)\delta_{x_ix_j}^2(x) + \theta_{\eps,x_j}^2(x)\delta_{x_i}^2(x) + \theta_{\eps,x_i}^2(x)\delta_{x_j}^2(x) + \theta_\eps(x)\theta_{\eps,x_ix_j}(x) \delta_{x_i}(x)\delta_{x_j}(x)   \notag\\
    & + \theta_\eps(x)\theta_{\eps,x_j}(x)\delta_{x_ix_j}(x)\delta_{x_i}(x) + \theta_\eps(x)\theta_{\eps,x_i}(x)\delta_{x_ix_j}(x)\delta_{x_j}(x) + \theta_{\eps(x), x_i}(x)\theta_{\eps,x_j}(x)\delta_{x_i}(x)\delta_{x_j}(x) \Big]   \notag\\
    & + \delta^{\frac{p-2}{p}+2\eps}(x) \Big[ \theta_\eps^2(x) \delta_{x_ix_j}(x)\delta_{x_i}(x)\delta_{x_j}(x) + \theta_\eps(x)\theta_{\eps,x_j}(x)\delta_{x_i}^2(x)\delta_{x_j}(x) + \theta_\eps(x)\theta_{\eps,x_i}(x)\delta_{x_j}^2(x)\delta_{x_i}(x) \Big]      \notag\\
    & + \delta^{\frac{3p-2}{p}+2\eps}(x) \Big[ \theta_\eps(x)\theta_{\eps,x_ix_j}(x)\delta_{x_ix_j}(x) + \theta_{\eps,x_ix_j}(x)\theta_{\eps,x_j}(x)\delta_{x_i}(x) + \theta_{\eps,x_ix_j}(x)\theta_{\eps,x_i}(x)\delta_{x_j}(x) \Big] \Bigg).    \label{ueps^2}
\end{align}
As in the case of $I_{1,\eps}$, notice that in $I_{2,\eps}$ we integrate only on the support of $\theta_\eps$, hence $I_{2,\eps}$ can be divided into three integrals on distinct regions, as follows:
\begin{align}
    J_{1,\eps} & = \int_{V_{2r}\setminus V_r}  \Big( \sum_{i,j} u_{\eps,x_ix_j}^2(x) \Big)^\frac{p}{2} - \left(\frac{2p-1}{p} + \eps\right)^p \left(\frac{p-1}{p} + \eps\right)^p \delta^{-1+p\eps}(x) \theta_\eps^p(x)  dx,    \label{J1eps}\\
    J_{2,\eps} & = \int_{\{\delta(x)>\eps\}\setminus V_{2r}}  \Big( \sum_{i,j} u_{\eps,x_ix_j}^2(x) \Big)^\frac{p}{2} - \left(\frac{2p-1}{p} + \eps\right)^p \left(\frac{p-1}{p} + \eps\right)^p \delta^{-1+p\eps}(x) \theta_\eps^p(x)  dx,     \label{J2eps}\\
    J_{3,\eps} & = \int_{\{\eps^2\leq \delta(x)\leq \eps\}}  \Big( \sum_{i,j} u_{\eps,x_ix_j}^2(x) \Big)^\frac{p}{2} - \left(\frac{2p-1}{p} + \eps\right)^p \left(\frac{p-1}{p} + \eps\right)^p \delta^{-1+p\eps}(x) \theta_\eps^p(x)  dx.      \label{J3eps}
\end{align}
By \eqref{delta-bound}, \eqref{deriv-delta-bound} and \eqref{deriv-theta-bound}, we deduce that $J_{1,\eps}$ is uniformly bounded as $\eps\to 0$. \\
Recall that $\theta_\eps(x)=1$ on $\{\delta(x)>\eps\}\setminus V_{2r}$, so restricting to this set in \eqref{ueps^2} we have
\begin{align}
     u_{\eps,x_ix_j}^2(x)  \leq \left( \frac{2p-1}{p}+\eps \right)^2\left( \frac{p-1}{p}+\eps \right)^2 \delta^{-\frac{2}{p}+2\eps}(x)\delta_{x_i}^2(x)\delta_{x_j}^2(x) &+ C \;\Big( \delta^{\frac{2p-2}{p}+2\eps}(x)\delta_{x_ix_j}^2(x)  \notag\\
     &  +  \delta^{\frac{p-2}{p}+2\eps}(x) \delta_{x_ix_j}(x)\delta_{x_i}(x)\delta_{x_j}(x) \Big).    \notag
\end{align}
Summing over $i,j$ we get
\begin{equation}
    \sum_{i,j=1}^N u_{\eps,x_ix_j}^2(x) \lesssim  \left( \frac{2p-1}{p}+\eps \right)^2\left( \frac{p-1}{p}+\eps \right)^2 \delta^{-\frac{2}{p}+2\eps}(x) + C\;\delta^{\frac{p-2}{p}+2\eps}(x),
\end{equation}
where $C=C(p, \|\delta\|_{L^\infty(\Om\setminus V_{2r})},\;\|D\delta\|_{L^\infty(\Om\setminus V_{2r})},\;\|D^2\delta\|_{L^\infty(\Om\setminus V_{2r})})$. Thus,
\begin{align}
    J_{2,\eps} \lesssim \int_{\{\delta(x)>\eps\}\setminus V_{2r}} \Bigg[\left( \frac{2p-1}{p}+\eps \right)^2\left( \frac{p-1}{p}+\eps \right)^2  & \delta^{-\frac{2}{p}+2\eps}(x) + C\;\delta^{\frac{p-2}{p}+2\eps}(x)\Bigg]^\frac{p}{2}    \notag\\  
    &-\left(\frac{2p-1}{p} + \eps\right)^p \left(\frac{p-1}{p} + \eps\right)^p \delta^{-1+p\eps}(x)  dx.
\end{align}
For $p\geq 4$, hence $\frac{p}{2}\geq 2$, using \eqref{ec:ineqshafrir1} for
$$a=\left( \frac{2p-1}{p}+\eps \right)^2\left( \frac{p-1}{p}+\eps \right)^2 \delta^{-\frac{2}{p}+2\eps}(x) \;\;\text{and}\;\; b=C\delta^{\frac{p-2}{p}+2\eps}(x),$$
we obtain, similar as in $T_{3,\eps}$, that
\begin{align}
    J_{2,\eps} &  \lesssim  \int_{\{\delta(x)>\eps\}\setminus V_{2r}} \abs{a}^{\frac{p}{2}-2}\abs{b}^2 dx + \int_{\{\delta(x)>\eps\}\setminus V_{2r}} \abs{a}^{\frac{p}{2}-1}\abs{b} dx + \int_{\{\delta(x)>\eps\}\setminus V_{2r}} \abs{b}^\frac{p}{2} dx    \notag\\
    & \lesssim  \int_{\{\delta(x)>\eps\}\setminus V_{2r}} \delta^{1+p\eps}(x) dx +  \int_{\{\delta(x)>\eps\}\setminus V_{2r}} \delta^{p\eps}(x) dx + \int_{\{\delta(x)>\eps\}\setminus V_{2r}} \delta^{\frac{p-2}{2}+p\eps}(x) dx    \notag\\
    & = \int_{\{\delta(x)>\eps\}\setminus V_{2r}} \delta^{p\eps}(x) \left[1 + \delta(x) + \delta^\frac{p}{2}(x)\right] dx   \notag\\
    & \lesssim \;\left|\{\delta(x)>\eps\}\setminus V_{2r}\right| < |\Om|,   \label{J2eps}
\end{align}
since $\delta$ is bounded in $\Om$, so $J_{2,\eps}$ is uniformly bounded for $p\geq 4$. When $2<p<4$, then $1<\frac{p}{2}<2$, so we apply \eqref{ec:ineqshafrir2} for the same choice of $a$ and $b$ and, similarly as in $T_{3,\eps}$, get
\begin{align}
     J_{2,\eps} &  \lesssim \int_{\{\delta(x)>\eps\}\setminus V_{2r}} \abs{a}^{\frac{p}{2}-1}\abs{b} dx + \int_{\{\delta(x)>\eps\}\setminus V_{2r}} \abs{b}^\frac{p}{2} dx,    \notag
\end{align}
which is uniformly bounded in view of \eqref{J2eps}. Similarly, for $1<p\leq 2$,\; $\frac{1}{2}<\frac{p}{2}\leq 1$, we apply \eqref{ec:ineqshafrir3} and we get
\begin{align}
    J_{2,\eps} &  \lesssim \int_{\{\delta(x)>\eps\}\setminus V_{2r}} \abs{b}^\frac{p}{2} dx,
\end{align}
which is again uniformly bounded. Therefore, $J_{2,\eps}$ is bounded for any $1<p<\infty$.\\
In order to justify the boundedness of $J_{3,\eps}$, we first estimate the sum of the terms $u_{\eps,x_ix_j}^2$ in $\{\eps^2\leq\delta(x)\leq\eps\}$. Using \eqref{deriv-1-thetaeps}, \eqref{deriv-2-thetaeps}, \eqref{deriv-theta-eps} and \eqref{ueps^2}, we have
\begin{align}
    \sum_{i,j=1}^N u_{\eps,x_ix_j}^2(x) & \lesssim \left( \frac{2p-1}{p}+\eps \right)^2\left( \frac{p-1}{p}+\eps \right)^2 \delta^{-\frac{2}{p}+2\eps}(x) + \delta^{\frac{4p-2}{p}+2\eps} \left(\frac{1}{\ln\left(\frac{1}{\eps}\right)} \frac{1}{\delta^2(x)} \right)^2    \notag\\
    & + \delta^{\frac{2p-2}{p}+2\eps} \left[ 1 + \left(\frac{1}{\ln\left(\frac{1}{\eps}\right)} \frac{1}{\delta(x)} \right)^2 + \left(\frac{1}{\ln\left(\frac{1}{\eps}\right)} \frac{1}{\delta^2(x)} \right) + \frac{1}{\ln\left(\frac{1}{\eps}\right)} \frac{1}{\delta(x)}  \right]    \notag\\
    & + \delta^{\frac{p-2}{2}+2\eps} \left[ 1 + \frac{1}{\ln\left(\frac{1}{\eps}\right)} \frac{1}{\delta(x)} \right] + \delta^{\frac{3p-2}{p}+2\eps} \left[ \frac{1}{\ln\left(\frac{1}{\eps}\right)} \frac{1}{\delta^2(x)} + \frac{1}{\ln^2\left(\frac{1}{\eps}\right)} \frac{1}{\delta^3(x)} \right]   \notag\\
    & \lesssim \left( \frac{2p-1}{p}+\eps \right)^2\left( \frac{p-1}{p}+\eps \right)^2 \delta^{-\frac{2}{p}+2\eps}(x) + \delta^{-\frac{2}{p}+2\eps} \left[\frac{1}{\ln\left(\frac{1}{\eps}\right)} + \frac{1}{\ln^2\left(\frac{1}{\eps}\right)} \right]    \notag\\
    & \lesssim \left( \frac{2p-1}{p}+\eps \right)^2\left( \frac{p-1}{p}+\eps \right)^2 \delta^{-\frac{2}{p}+2\eps}(x) + \delta^{-\frac{2}{p}+2\eps} \frac{1}{\ln\left(\frac{1}{\eps}\right)}.
\end{align}
Therefore, 
\begin{align}
    J_{3,\eps} & \lesssim \int_{\{\eps^2\leq\delta(x)\leq \eps\}} \Bigg| \left( \frac{2p-1}{p}+\eps \right)^2\left( \frac{p-1}{p}+\eps \right)^2 \delta^{-\frac{2}{p}+2\eps}(x) + \delta^{-\frac{2}{p}+2\eps}(x) \frac{1}{\ln\left(\frac{1}{\eps}\right)} \Bigg|^\frac{p}{2} \notag\\  &\quad\quad\quad\quad\quad\quad\quad\quad\quad\quad\quad\quad\quad\quad\quad\quad\quad\quad\quad\quad-\left(\frac{2p-1}{p} + \eps\right)^p \left(\frac{p-1}{p} + \eps\right)^p \delta^{-1+p\eps}(x)  dx   \notag\\
    & = \int_{\{\eps^2\leq\delta(x)\leq \eps\}} \delta^{-1+p\eps} (x) \Bigg[ \Bigg| \left( \frac{2p-1}{p}+\eps \right)^2\left( \frac{p-1}{p}+\eps \right)^2 + \frac{1}{\ln\left(\frac{1}{\eps}\right)} \Bigg|^\frac{p}{2}   \notag\\
    &\quad\quad\quad\quad\quad\quad\quad\quad\quad\quad\quad\quad\quad\quad\quad\quad\quad\quad\quad\quad - \left(\frac{2p-1}{p} + \eps\right)^p \left(\frac{p-1}{p} + \eps\right)^p \Bigg] dx. \notag
\end{align}
Using again \eqref{ec:ineqshafrir1}, \eqref{ec:ineqshafrir2}, \eqref{ec:ineqshafrir3} for $p\geq 4$, $2<p<4$ and $1<p\leq 2$ respectively, applied for
\begin{equation}
    a=\left(\frac{2p-1}{p} + \eps\right)^p \left(\frac{p-1}{p} + \eps\right)^p \;\;\text{and}\;\; b= \frac{1}{\ln\left(\frac{1}{\eps}\right)}
\end{equation}
similar as for $J_{2,\eps}$, we get that 
\begin{align}
    J_{3,\eps} & \lesssim \frac{1}{\ln\left(\frac{1}{\eps}\right)}\int_{\{\eps^2\leq\delta(x)\leq \eps\}} \delta^{-1+p\eps}(x) dx,
\end{align}
which is bounded for any $1<p<\infty$, by \eqref{limita-eps}. Finally, since $I_{1,\eps},I_{2,\eps}$ are bounded as $\eps$ goes to $0$, in view of \eqref{int-infinity} and \eqref{ec:convergence1}, we have that
\begin{align}
      \frac{\int_\Om \abs{\hess u_\eps}^p dx}{\int_\Om \frac{\abs{\nabla u_\eps}^p}{\delta^p(x)} dx}  = \frac{\left(\frac{2p-1}{p} + \eps\right)^p \left(\frac{p-1}{p} + \eps\right)^p \int_\Om \delta^{-1+p\eps} \theta_\eps^p dx + I_{1,\eps}}{\left(\frac{2p-1}{p} + \eps\right)^p \int_\Om \delta^{-1+p\eps}\theta_\eps^p dx + I_{2,\eps}}    \;\;\stackrel{\eps\to 0}{\longrightarrow} \;\;\left(\frac{p-1}{p} \right)^p,
\end{align}
which proves the optimality of the constant in \eqref{ec:hessianineq2} when $-\Delta\delta\geq 0$ in $\Om\setminus\ridge$. The proof of the theorem is now finished.
\end{proof}

\begin{remark} 
    The same minimizing sequence used in the previous proof might be used to justify the first inequality of \eqref{constCZ}. The computations are carried in the same manner, the difference being that we restrict to the Laplacian instead of the Hessian of $u_\eps$, which makes it easier to estimate the integrals containing second order derivatives.
\end{remark}

\begin{proof}[Proof of Theorem \ref{thmExtension}]
The computations are carried out as in the proof of Theorem \ref{thm1}, only taking into account the contribution of the boundary of $\Om_\eps$. For $u\in C^\infty_c(\Om)$ we have:
\begin{align}
    \int_{\Om\setminus\Om_\eps} \frac{\abs{\nabla u}^p}{\delta^p} &  = \frac{1}{p-1} \int_{\Om\setminus\Om_\eps} \delta^{1-p} \nabla\delta\nabla(\abs{\nabla u}^p) dx + \frac{1}{p-1} \int_{\Om\setminus\Om_\eps} \frac{\abs{\nabla u}^p}{\delta^{p-1}(x)} \Delta\delta dx   \notag\\
    & - \frac{1}{p-1} \int_{\prt\Om\eps} \frac{|\nabla u|^p}{\delta^{p-1}(x)} (\nabla\delta \cdot\nu_\eps) d\sigma   \notag\\
    & \leq \frac{p}{p-1} \int_{\Om\setminus\Om_\eps} \delta^{1-p} \abs{\nabla u}^{p-1} |\;\hess u| \; dx + \frac{1}{p-1} \int_{\Om\setminus\Om_\eps} \frac{\abs{\nabla u}^p}{\delta^{p-1}(x)} \Delta\delta dx.  \notag
\end{align}
Again, by applying the Young Inequality we get
\begin{equation}   \label{ineqOmegaeps}
        \int_{\Om\setminus\Om_\eps} \abs{\hess u}^p dx \geq \left(\frac{p-1}{p} \right)^p \int_{\Om\setminus\Om_\eps} \frac{\abs{\nabla u}
        ^p}{\delta^p(x)} dx + \left(\frac{p-1}{p} \right)^{p-1} \int_{\Om\setminus\Om_\eps} \frac{\abs{\nabla u}
        ^p}{\delta^{p-1}(x)} (-\Delta\delta) dx,    
\end{equation}
Since $\hess u\in L^p(\Om)$, we have
\begin{equation}
    \int_{\Om\setminus\Om_\eps} |\hess u|^p dx \to  \int_{\Om} |\hess u|^p dx, \;\;\text{as}\;\;\eps\to 0.    \label{convg-hess}
\end{equation}
Then, by Fatou's Lemma, \eqref{ineqOmegaeps} and \eqref{convg-hess}, we successively obtain
\begin{align}
    \left(\frac{p-1}{p} \right)^p & \int_{\Om} \frac{\abs{\nabla u}
        ^p}{\delta^p(x)} dx  + \left(\frac{p-1}{p} \right)^{p-1} \int_{\Om} \frac{\abs{\nabla u}
        ^p}{\delta^{p-1}(x)} (-\Delta\delta) dx      \notag\\
        & = \left(\frac{p-1}{p} \right)^p \int_{\Om} \liminf_{\eps\to 0} \frac{\abs{\nabla u}
        ^p}{\delta^p(x)} \chi_{\Om\setminus\Om_\eps} dx + \left(\frac{p-1}{p} \right)^{p-1} \int_{\Om} \liminf_{\eps\to 0} \frac{\abs{\nabla u}
        ^p}{\delta^{p-1}(x)} (-\Delta\delta) \chi_{\Om\setminus\Om_\eps}dx   \notag\\ 
        & \leq \liminf_{\eps\to 0} \left(\frac{p-1}{p} \right)^p \int_{\Om\setminus\Om_\eps} \frac{\abs{\nabla u}
        ^p}{\delta^p(x)} dx + \liminf_{\eps\to 0} \left(\frac{p-1}{p} \right)^{p-1} \int_{\Om\setminus\Om_\eps} \frac{\abs{\nabla u}
        ^p}{\delta^{p-1}(x)} (-\Delta\delta) dx    \notag\\
        & \leq \liminf_{\eps\to 0} \left\{ \left(\frac{p-1}{p} \right)^p \int_{\Om\setminus\Om_\eps} \frac{\abs{\nabla u}
        ^p}{\delta^p(x)} dx + \left(\frac{p-1}{p} \right)^{p-1} \int_{\Om\setminus\Om_\eps} \frac{\abs{\nabla u}
        ^p}{\delta^{p-1}(x)} (-\Delta\delta) dx \right\}    \notag\\
        & = \liminf_{\eps\to 0} \int_{\Om\setminus\Om_\eps} |\hess u|^p dx   \notag\\
        & = \int_\Om |\hess u|^p dx.   \notag
\end{align}
\end{proof}

\section{Second order inequalities with the Laplacian}
\begin{proof}[Proof of Theorem \ref{thm3}]
Let $u(x)=v(\delta(x))$. By $v'$ we denote the partial derivative with respect to $\delta$, i.e. $\frac{\prt v}{\prt \delta}$. Again, by Eikonal equation $|\nabla\delta|=1$, we have the following identity 
$$\delta^{-p}=\frac{1}{1-p}\nabla\left(\delta^{1-p}\right)\nabla\delta.$$
Take into account that, by the formula \eqref{ec:s-laplacian}, the Laplacian of $u$ is $\Delta u = v''(\delta) + \Delta\delta v'(\delta)$. Then, integrating by parts and applying Young's Inequality $ab\leq \frac{1}{p}a^p+\frac{p-1}{p}b^\frac{p}{p-1}$ with $a=\frac{p}{p-1}\left[v''+\Delta\delta v'\right]$ and $b=\abs{v'}^{p-2}v'\delta^{1-p}$, we get
\begin{align}
    \int_\Om \frac{\abs{\nabla u}^p}{\delta^p(x)} dx & = \frac{1}{1-p} \int_\Om \abs{v'}^p \nabla\left(\delta^{1-p}(x)\right)\nabla\delta dx  \notag\\
    & = \frac{p}{p-1} \int_\Om \abs{v'}^{p-2} v' v''\delta^{1-p}(x) dx + \frac{1}{p-1} \int_\Om \abs{v'}^p \delta^{1-p}(x) \Delta\delta dx  \notag\\
    & = \frac{p}{p-1} \int_\Om \abs{v'}^{p-2} v' \left[v''+\Delta\delta v'\right] \delta^{1-p}(x) dx -  \int_\Om \abs{v'}^p \delta^{1-p}(x) \Delta\delta dx  \notag\\
    & \leq \frac{1}{p} \left(\frac{p}{p-1}\right)^p \int_\Om \left[v''+\Delta\delta v'\right]^p dx + \frac{p-1}{p} \int_\Om \frac{\abs{v'}^p}{\delta^p(x)} dx - \int_\Om \abs{v'}^p \delta^{1-p}(x) \Delta\delta dx  \notag\\
    & = \frac{1}{p} \left(\frac{p}{p-1}\right)^p \int_\Om \abs{\Delta u}^p dx + \frac{p-1}{p} \int_\Om \frac{\abs{\nabla u}^p}{\delta^p(x)} dx - \int_\Om \abs{\nabla u}^p \delta^{1-p}(x) \Delta\delta dx. \notag
\end{align}
After an algebraic reorganization, we conclude that
\begin{equation}
    \int_\Om \abs{\Delta u}^p dx \geq \left(\frac{p-1}{p}\right)^p \int_\Om \frac{\abs{\nabla u}^p}{\delta^p(x)} dx + p\left(\frac{p-1}{p}\right)^p \int_\Om \frac{\abs{\nabla u}^p}{\delta^{p-1}(x)} \Delta\delta dx.
\end{equation}
\end{proof}

\section{Existence of non-trivial solutions of \eqref{ec:systemP}}

We will prove Theorem \ref{thmExistence} using critical point theory. The proof will be split into two parts: for the first part of the theorem we use the direct method of calculus of variations, while for the second part we use the Mountain Pass Theorem. First, let us recall some definitions and set up the framework.\\
For any $\lmb<\lmb^\sharp(p, \Om, Z=\frac{1}{\delta^p})$ consider the functional $I_\lmb:W^{2,p}_0(\Om)\to\R$,
\begin{equation}   \label{ec:defI}
    I_\lmb[u]:= \frac{1}{p} \left( \int_\Om \abs{\Delta u}^p dx - \lmb\int_\Om \frac{\abs{\nabla u}^p}{\delta^p(x)}dx \right) - \frac{1}{q} \int_\Om \abs{u}^{q}dx.   
\end{equation}
The Fréchet derivative of $I_\lmb$ at $u$ is defined for any $\phi\in W^{2,p}_0(\Om)$ by 
\begin{equation}   \label{ec:defI'}
    I_\lmb'[u]\phi = \int_\Om \abs{\Delta u}^{p-2}\Delta u \Delta\phi dx - \lmb\int_\Om \frac{\abs{\nabla u}^{p-2}}{\delta^p(x)}\nabla u\cdot\nabla\phi dx - \int_\Om \abs{u}^{q-2}u\phi dx.   
\end{equation}
Clearly, $I_\lmb\in C^1\left(W^{2,p}_0(\Om),\R\right)$, meaning that the Fréchet derivative of $I_\lmb$ exists and is continuous on $W^{2,p}_0(\Om)$. According to the definition \eqref{weak-sol}, we may look for weak solutions of the problem \eqref{ec:systemP} as critical points for $I_\lmb$. We will need the following compacity condition.
\begin{definition}[Palais-Smale condition, \cite{mawhin}]
    Let $X$ be a Banach space and $I\in C^1(X,\R)$. We say that $I$ satisfies the Palais-Smale condition if any sequence $(u_n)_n\subset X$ which satisfies the conditions
    \begin{enumerate}
        \item[i)] $\left(I[u_n]\right)_n$ is bounded and
        \item[ii)] $I'[u_n] \to 0$ in $X'$
    \end{enumerate}
    contains a convergent subsequence.
\end{definition}
For the sake of completeness, let us recall the Mountain Pass Theorem.
\begin{theorem}[Mountain Pass, \cite{evansCarte}]   \label{thmMP}
    Let $X$ be a Banach space and $I\in C^1(X,\R)$ with the following properties:
    \begin{enumerate}
        \item[i)] $I[0]=0$;
        \item[ii)] there exist constants $r,\rho>0$ such that if $\|u\|_X=\rho$ then $I[u]\geq r$;
        \item[iii)] there exists an element $v\in X$ with $\|v\|_X>\rho$ and $I[v]\leq 0$;
        \item[iv)] $I$ satisfies the Palais-Smale condition. 
    \end{enumerate}
    Define $\Gamma:=\{ g\in C([0,1],X) \;|\; g(0)=0,\; g(1)=v \}$. Then the value 
    $$c=\inf_{g\in\Gamma} \max_{0\leq t\leq 1} I[g(t)]$$
    is a critical value of $I$. 
\end{theorem}

\subsection{Proof of Theorem \ref{thmExistence} - Case 1.}
Let $I:= \inf_{\phi\in\W} I_\lmb[\phi] <\infty$. Let $(\phi_n)_n$ in $\W$ such that $I_\lmb[\phi_n]\searrow I$. We will prove that $\phi$ is a minimum for $I$. We need the following lemma.
\begin{lemma}
    For any $\lmb<\lmb^\sharp(p, \Om, Z=\frac{1}{\delta^p})$ the functional $I_\lmb$ is coercive.
\end{lemma}
\begin{proof}
    By the Hardy-Rellich inequality \eqref{HR} and embedding \eqref{embed} we have
    \begin{align}
        I_\lmb[\phi] & = \frac{1}{p}\|\Delta\phi\|_{L^p(\Om)}^p - \frac{\lmb}{p} \left\|\frac{\nabla\phi}{\delta}\right\|_{L^p(\Om)}^p - \frac{1}{q}\|\phi\|_{L^q(\Om)}^q    \notag\\
        & \geq \frac{1}{p}\left(1-\frac{\lmb}{\lmb^\sharp(p, \Om, Z=\frac{1}{\delta^p})}\right) \|\Delta\phi\|_{L^p(\Om)}^p - \frac{1}{q}\|\phi\|_{L^q(\Om)}^q    \notag\\
        & \geq \frac{1}{p}\left(1-\frac{\lmb}{\lmb^\sharp(p, \Om, Z=\frac{1}{\delta^p})}\right) \|\phi\|_{\W}^p - \frac{C}{q}\|\phi\|_{\W}^q    \notag\\
        & = \|\phi\|_{\W}^q \left(C_1\|\phi\|_{\W}^{p-q}-\frac{1}{q}\right).   \notag
    \end{align}
    Since $q<p$, we get the conclusion.
\end{proof}
This implies that the sequence $(\phi_n)_n$ is bounded in $\W$ and therefore it converges weakly to a function $\phi_0\in\W$. Now we prove that $\phi_0$ is a minimizer for $I_\lmb$. Note that
\begin{equation}
    I_\lmb[\phi_0] = \frac{1}{p}\|\phi_0\|^p - \frac{1}{q} \|\phi_0\|_{L^q(\Om)}^q.  \notag
\end{equation}
By the weak lower semi-continuity of the norm, we get:
\begin{align}
    I\leq I_\lmb[\phi_0] & = \frac{1}{p}\|\phi_0\|^p - \frac{1}{q} \|\phi_0\|_{L^q(\Om)}^q    \notag\\
    & \leq \frac{1}{p}\liminf_{n\to\infty}\|\phi_n\|^p - \frac{1}{q} \lim_{n\to\infty} \|\phi_n\|_{L^q(\Om)}^q     \notag\\
    & \leq \liminf_{n\to\infty} \left(\frac{1}{p}\|\phi_n\|^p - \frac{1}{q} \|\phi_n\|_{L^q(\Om)}^q\right)   \notag\\
    & = \liminf_{n\to\infty} I_\lmb[\phi_n] = \lim_{n\to\infty} I_\lmb[\phi] = I.   \notag
\end{align}
Hence, $I_\lmb[\phi_0]=I$ and we can deduce that $\phi_0$ is a weak solution for \eqref{ec:systemP}. To show that the solution is non-trivial, let $w\in\C$, $w\neq 0$ and $t>0$. Then
    \begin{align}
        I_\lmb[tw] & = \frac{1}{p}\int_\Om \abs{\Delta(tw)}^p dx - \frac{\lmb}{p} \int_\Om \frac{\abs{\nabla(tw)}^p}{\delta^p(x)} dx - \frac{1}{q}\int_\Om \abs{tw}^{q} dx   \notag\\
        & = \frac{t^p}{p} \|w\|^p - \frac{t^{q}}{q}\|w\|_{L^q(\Om)}^p \leq t^q \left(C_1 t^{p-q} - C_2\right)   \notag
    \end{align}
Thus, for $q<p$, there exists $t_0>0$, say $t_0<\left(\frac{C_2}{C_1}\right)^\frac{1}{p-q}$, such that for any $t<t_0$ we have $I_\lmb[w]<0$. We get by the definition of $I$ that $I=I_\lmb[\phi_0]<0$, so $\phi_0$ is not trivial.

\subsection{Proof of Theorem \ref{thmExistence} - Case 2.}
We check the hypotheses of the Mountain Pass Theorem. Clearly, $I_\lmb[0]=0$. Denote by $\lmb^+$ and $\lmb^-$ the positive part the negative part, respectively, of $\lmb<\lmb^\sharp(p,\Om, Z=\frac{1}{\delta^p})$.
\begin{lemma}  \label{lemma1MP}
    For any $\lmb<\lmb^\sharp(p,\Om, Z=\frac{1}{\delta^p})$ there exists $r,\rho>0$ such that $I_\lmb[u]\geq r$ for any $u\in W^{2,p}_0(\Om)$ with $\|u\|_{\W}=\rho$.
\end{lemma}
\begin{proof}
    Using the Hardy-Rellich inequality \eqref{HR} and the Sobolev embedding $\W\subset L^{q}(\Om)$ we get 
    \begin{align}
        I_\lmb[u] & \geq \frac{1}{p}\left(1-\frac{\lmb^+}{\lmb^\sharp(p, \Om, Z=\frac{1}{\delta^p})}\right)\int_\Om \abs{\Delta u}^p dx - \frac{1}{q}\|u\|_{L^q(\Om)}^q \notag\\
        & \geq \frac{1}{p}\left(1-\frac{\lmb^+}{\lmb^\sharp(p, \Om, Z=\frac{1}{\delta^p})}\right)\|u\|_{\W}^p - \frac{C}{q}\|u\|_{\W}^{q}    \notag
    \end{align} 
    For $\|u\|_{\W}=\rho>0$, we deduce that
    $$I_\lmb[u]\geq C_1\rho^p - C_2\rho^{q}=\rho^p(C_1-C_2\rho^{q-p}).$$
    Since $p<q$, we can find $\rho$ sufficiently small $(\rho<(C_1/C_2)^{1/(q-p)})$ such that there exists $r=r(\rho)=\rho^p(C_1-C_2\rho^{q-p})$ with $I_\lmb[u]\geq r>0$. 
\end{proof}
\begin{lemma}   \label{lemma2MP}
    For any $\lmb<\lmb^\sharp(p,\Om, Z=\frac{1}{\delta^p})$ there exists $\rho>0$ and $v\in\W$ such that $\|v\|_{\W}>\rho$ and $I_\lmb[v]<0$.
\end{lemma}
\begin{proof}
    Let $w\in\C$, $w\neq 0$ and $t>0$. Then, by the inequality \eqref{HR} and embedding \eqref{embed} we have
    \begin{align}
        I_\lmb[tw] & = \frac{1}{p}\int_\Om \abs{\Delta(tw)}^p dx - \frac{\lmb}{p} \int_\Om \frac{\abs{\nabla(tw)}^p}{\delta^p(x)} dx - \frac{1}{q}\int_\Om \abs{tw}^{q} dx   \notag\\
        & \leq \frac{t^p}{p}\left(1+\frac{\lmb^-}{\lmb^\sharp(p,\Om, Z=\frac{1}{\delta^p})}\right) \|w\|_{\W}^p - \frac{t^{q}}{q}\|w\|_{L^q(\Om)}^p \;\; \stackrel{t\to\infty}{\longrightarrow} -\infty.   \notag
    \end{align}
    Since $p<q$, there exists $t_0>0$ and $\rho>0$ such that for any $t>t_0$ we can pick $v\in\W$ with $\|v\|_{\W}=t\|w\|_{\W}\geq t_0\|w\|_{\W} (=\rho)$ and $I_\lmb[v]<0$.   
\end{proof}
With these two lemmas, we showed that the functional $I_\lmb[.]$ has the geometry of the Mountain Pass.
\begin{lemma}   \label{ec:lemma3MP}
    For any $\lmb\leq 0$ and any $1<p<\frac{N}{2}$, the functional $I_\lmb$ satisfies the Palais-Smale condition.
\end{lemma}
\begin{proof}
    Let $(u_n)_n\subset\W$ be a Palais-Smale sequence, i.e. $\left(I_\lmb[u_n]\right)_n$ is bounded in $\R$ and $I_\lmb'[u_n]\to 0$ in $\left(\W\right)'$ - the second condition means that $I'_\lmb[u_n]\phi \to 0$, for any $\phi\in \W$. Therefore, there exist $C_1, C_2$ positive constants such that $I_\lmb[u_n]\leq C_1$ and $\|I'_\lmb[u_n]\|_{\left( \W \right)'} \leq C_2$. We prove first that $(u_n)_n$ is bounded in $\W$ \text. Again, by \eqref{HR},
    \begin{align}
        qI_\lmb[u_n] + \abs{I_\lmb'[u_n]u_n} & \geq q I_\lmb[u_n] - I_\lmb'[u_n]u_n   \notag\\
        & =\left(\frac{q}{p}-1\right) \left( \int_\Om \abs{\Delta u_n}^p dx - \lmb\int_\Om \frac{\abs{\nabla u_n}^p}{\delta^p(x)} dx \right)  \notag\\
        & \geq \left(\frac{q}{p}-1\right)\left(1-\frac{\lmb}{\lmb^\sharp(p, \Om, Z=\frac{1}{\delta^p})}\right) \|u_n\|_{\W}^p \;:=\; \Tilde{C} \|u_n\|_{\W}^p.   \notag
    \end{align}
    On the other hand, 
    \begin{equation}
        qI_\lmb[u_n] + \abs{I_\lmb'[u_n]u_n} \leq q C_1 + C_2 \|u_n\|_{\W}.   \notag
    \end{equation}
    Thus, we get that
    \begin{equation}
        \Tilde{C}\|u_n\|_{\W} \left( \|u_n\|_{\W}^{p-1} - 1\right) \leq q C_1.   \notag
    \end{equation}
    Hence, $(u_n)_n$ is bounded in $\W$. Therefore, taking into account the compact embedding $\W\subset L^r(\Om)$, for any $1\leq r<p^{**}$ and $1<p<\frac{N}{2}$, there exists a subsequence (denoted also by $(u_n)_n$ from now on, every time we pass to a subsequence) and $u\in\W$ satisfying 
    \begin{align}
        u_n & \rightharpoonup u \;\; \text{in }\;\W, \;\; (\text{due to $\W$ being reflexive; see, e.g.},\cite[\text{Theorem 3.18}]{brezis-carte})   \label{weak-conv}\\
        u_n & \to u \;\;\text{in }\;L^{r}(\Om), \;\;(\text{Rellich-Kondrachov Theorem, see, e.g., \cite[Theorem 6.3]{adams}}.  \label{Lr-conv}  
    \end{align}
By the weak convergence of $(u_n)_n$, we know that $\varphi(u_n)\to\varphi(u)$, for any $\varphi\in \left(\W\right)'$. Since $I_\lmb'[u]$ is in $\left(\W\right)'$, we can take $\varphi=I_\lmb'[u]$, and we obtain that $I_\lmb'[u](u_n-u)\to 0$.\\
Then, considering that $I'_\lmb[u_n]\to 0$ in $\left(\W\right)'$ we have
\begin{equation}
     \left(I_\lmb'[u_n] - I_\lmb'[u]\right)(u_n-u) \to 0 \;\;\; \text{when}\;\;n\to\infty.    \notag
\end{equation}
By this and \eqref{ec:defI'} we also have
\begin{align}    \label{diffI'}
    \left(I_\lmb'[u_n] - I_\lmb'[u]\right)(u_n-u) &= I_\lmb'[u_n](u_n-u) - I_\lmb'[u](u_n-u)  \notag\\
    & = \int_\Om \left[\abs{\Delta u_n}^{p-2}\Delta u_n - \abs{\Delta u}^{p-2}\Delta u\right](\Delta u_n-\Delta u) dx  \notag\\
    & - \lmb\int_\Om \left[ \frac{\abs{\nabla u_n}^{p-2}\nabla u_n}{\delta^p(x)} - \frac{\abs{\nabla u}^{p-2}\nabla u}{\delta^p(x)} \right](\nabla u_n-\nabla u) dx   \notag\\
    & - \int_\Om \left[\abs{u_n}^{q-2}u_n-\abs{u}^{q-2}u\right](u_n-u) dx  \;\;\to\;0\;\; \text{when}\;\;n\to\infty.
\end{align}
We prove now that $(u_n)_n$ converges strongly to $u$ in $\W$, i.e. 
\begin{equation}
    ||u_n-u||_{\W}^p = \int_\Om \abs{\Delta u_n-\Delta u}^p dx \;\;\to\; 0, \;\; \text{when} \; n\to\infty.    \notag
\end{equation}
The last term in \eqref{diffI'} can be estimated by the Hölder's inequality as follows:
\begin{align}    \label{3rd-term}
    \Bigg|  \int_\Om \left[\abs{u_n}^{q-2}u_n-\abs{u}^{q-2}u\right]& (u_n-u) dx \Bigg|    \leq \left|  \int_\Om \abs{u_n}^{q-2}u_n(u_n-u) dx \right| + \left| \int_\Om \abs{u}^{q-2}u(u_n-u) dx \right|   \notag\\
    & \leq \left(\int_\Om |u_n|^q\right)^\frac{q-1}{q}  \left(\int_\Om |u_n-u|^q\right)^\frac{1}{q} +  \left(\int_\Om |u|^q\right)^\frac{q-1}{q} \left(\int_\Om |u_n-u|^q\right)^\frac{1}{q}  \notag\\
    &\;\;\to\; 0 \;\; \text{when}\; n\to\infty,  
\end{align}
since $u_n\to u$ in $L^q(\Om)$ by \eqref{Lr-conv}. Recalling that $\lmb\leq 0$ we can conclude, by \eqref{diffI'}, that
\begin{equation}
    \int_\Om \left[\abs{\Delta u_n}^{p-2}\Delta u_n - \abs{\Delta u}^{p-2}\Delta u\right](\Delta u_n-\Delta u) dx
    - \int_\Om \left[\abs{u_n}^{q-2}u_n-\abs{u}^{q-2}u\right](u_n-u) dx  \;\;\stackrel{n\to\infty}{\longrightarrow}\;0.    \notag
\end{equation}

Combining this with \eqref{3rd-term}, we get that
\begin{equation}     \label{convg-lapl}
     \int_\Om \left[\abs{\Delta u_n}^{p-2}\Delta u_n - \abs{\Delta u}^{p-2}\Delta u\right](\Delta u_n-\Delta u) dx \;\;\to\;0 \;\;\text{when}\; n\to\infty.
\end{equation}
For $p\geq2$, using \eqref{ineqlindqvist>2}, we have 
\begin{equation}
    \int_\Om \abs{\Delta u_n-\Delta u}^p dx \leq 2^{p-2} \int_\Om \left[\abs{\Delta u_n}^{p-2}\Delta u_n - \abs{\Delta u}^{p-2}\Delta u\right](\Delta u_n-\Delta u) dx,   \notag
\end{equation}
which converges to $0$ by \eqref{convg-lapl}, so $u_n$ converges strongly to $u$ in $\W$ for $p\geq 2$.\\
For $1<p<2$, we use Hölder's inequality (with $\frac{2}{p}$ and $\frac{2}{2-p}$) and we get
\begin{align}
    \int_\Om \abs{\Delta u_n-\Delta u}^p dx & = \int_\Om \frac{\abs{\Delta u_n-\Delta u}^p}{(1+\abs{\Delta u_n}^2+\abs{\Delta u}^2)^\frac{p(2-p)}{4}} (1+\abs{\Delta u_n}^2+\abs{\Delta u}^2)^\frac{p(2-p)}{4}dx   \notag\\
    & \leq \left( \int_\Om \frac{\abs{\Delta u_n-\Delta u}^2}{(1+\abs{\Delta u_n}^2+\abs{\Delta u}^2)^\frac{2-p}{2}} dx \right)^\frac{p}{2} \left( \int_\Om (1+\abs{\Delta u_n}^2+\abs{\Delta u}^2)^\frac{p}{2} dx \right)^\frac{2-p}{2}.   \notag
\end{align}
By raising to the power $\frac{2}{p}$ and applying \eqref{ineqlindqvist<2} we obtain
\begin{align}
    \|&u_n-u\|_{\W}^2  =  \left(\int_\Om \abs{\Delta u_n-\Delta u}^p dx \right)^{\frac{2}{p}}    \notag\\
    &\leq \int_\Om \frac{\abs{\Delta u_n-\Delta u}^2}{(1+\abs{\Delta u_n}^2+\abs{\Delta u}^2)^\frac{2-p}{2}} dx  \;\left( \int_\Om (1+\abs{\Delta u_n}^2+\abs{\Delta u}^2)^\frac{p}{2} dx \right)^\frac{2-p}{p}   \notag\\
    & \leq \frac{1}{p-1} \int_\Om \left[\abs{\Delta u_n}^{p-2}\Delta u_n - \abs{\Delta u}^{p-2}\Delta u\right](\Delta u_n-\Delta u) dx  \;\left( \int_\Om (1+\abs{\Delta u_n}^2+\abs{\Delta u}^2)^\frac{p}{2} dx \right)^\frac{2-p}{p},  \notag
\end{align}
which again converges to $0$ by \eqref{convg-lapl}, so $u_n$ converges strongly to $\W$ also for $1<p<2$. Hence, the lemma is proved.
\end{proof}

\begin{proof}[Proof of Theorem \ref{thmExistence}-2]
    By the lemmas above, all the hypotheses of Theorem \eqref{thmMP} are satisfied. Hence, there exists a critical point $\varphi_\lmb$, characterized by
    $$I_\lmb[\varphi_\lmb] = \inf_{\gamma\in\Gamma} \max_{t\in [0,1]} I_\lmb[\gamma(t)] \;\;\; \text{and}\;\;\; I_
    \lmb'[\varphi_\lmb]=0.$$
    Therefore, $\varphi_\lmb$ is a non-trivial weak solution to the problem \eqref{ec:systemP}. 
\end{proof}

\section{Non-existence of non-trivial solutions of \eqref{ec:systemP}}
We prove the nonexistence via a Pohozaev-type identity.
\begin{lemma}
    Let $u$ be a solution of \eqref{ec:systemP}. Then the following identity holds:
    \begin{align}    \label{pohozaev-0}
    \left(\frac{p-1}{p}\right) & \int_{\prt\Om} \abs{\Delta u}^{p} (x\cdot \n) \;d\sigma + \left(\frac{N-2p}{p}\right) \int_\Om \abs{\Delta u}^p dx - \lmb\int_\Om \frac{\abs{\nabla u}^p}{\delta^{p+1}} (x\cdot\n(N(x))) dx \notag\\
    & = \lmb\left(\frac{N-p}{p}\right) \int_\Om \frac{\abs{\nabla u}^p}{\delta^p} dx + \frac{N}{q} \int_\Om |u|^q dx,
\end{align}
    where $\n$ is the outward normal at $\prt\Om$.
\end{lemma}
 
\begin{proof}
    We multiply the equation in \eqref{ec:systemP} by $(x\cdot\nabla u)$ and integrate on $\Om$:
    \begin{equation}  \label{pohozaev-1}
        \int_\Om \Delta(\abs{\Delta u}^{p-2}\Delta u) (x\cdot\nabla u) dx + \lmb \int_\Om \dvg\left(\frac{\abs{\nabla u}^{p-2}\nabla u}{\delta^p}\right) (x\cdot\nabla u) dx = \int_\Om |u|^{q-2}u (x\cdot\nabla u) dx.
    \end{equation}
    Since $u=0$ on $\prt\Om$, the last integral is easily computed as:
    \begin{equation}   \label{int-poho-1}
        \int_\Om |u|^{q-2}u (x\cdot\nabla u) dx = \frac{1}{q} \int_\Om x\cdot \nabla(|u|^q) dx = -\frac{N}{q} \int_\Om |u|^q dx.     
    \end{equation}
    In order to compute the first integral in \eqref{pohozaev-1}, we use the next identity, which was proved in \cite[Proposition 2.1]{mitidieri}:
    \begin{align}  \label{mitidieri}
        \int_\Om \Big\{ \Delta v (x\cdot\nabla u) + \Delta u (x\cdot\nabla v) \Big\} dx & = \int_{\prt\Om} \Big\{ \frac{\prt v}{\prt \n} (x\cdot\nabla u) + \frac{\prt u}{\prt \n} (x\cdot\nabla v) - (\nabla u \cdot \nabla v)(x\cdot\n) \Big\} \;d\sigma  \notag\\
        & + (N-2) \int_\Om \nabla u\cdot \nabla v dx, \;\;\;\; \text{for any}\;\; u,v \in \W.
    \end{align}
We apply this for $v=\abs{\Delta u}^{p-2}\Delta u \stackrel{not}{=}w$ and the first integral in \eqref{pohozaev-1} becomes
\begin{align}    \label{pohozaev-2}
     \int_\Om \Delta(\abs{\Delta u}^{p-2}\Delta u) (x\cdot\nabla u) dx & = \int_{\prt\Om} \Big\{ \frac{\prt w}{\prt \n} (x\cdot\nabla u) + \frac{\prt u}{\prt \n} (x\cdot\nabla w) - (\nabla u \cdot \nabla w)(x\cdot\n) \Big\} \;d\sigma  \notag\\
    & - \int_\Om \Delta u (x\cdot\nabla w) dx + (N-2) \int_\Om \nabla u\cdot \nabla w dx. 
\end{align}
Due to $u=\frac{\prt u}{\prt\n}=0$ on $\prt\Om$ and $\nabla u=\frac{\prt u}{\prt \n}\;\n$ on the boundary, then the gradient of $u$ is zero on the boundary. Thus, the boundary integrals above all vanish. By the divergence theorem,
\begin{align}
    \int_\Om \Delta u (x\cdot\nabla w) dx 
    & =  \int_{\prt\Om} \abs{\Delta u}^{p} (x\cdot \n) \;d\sigma - N\int_\Om \abs{\Delta u}^p dx - \frac{1}{p}\int_\Om x \cdot \nabla\left(\abs{\Delta u}^p\right) \; dx     \notag\\
    & = \left(\frac{p-1}{p}\right) \int_{\prt\Om} \abs{\Delta u}^{p} (x\cdot \n) \;d\sigma - \frac{N(p-1)}{p} \int_\Om \abs{\Delta u}^p.   \notag
\end{align}
Using that $\frac{\prt u}{\prt\n}=0$ on $\prt\Om$ and Green's formula, the last integral in \eqref{pohozaev-2} is
\begin{equation}
    \int_\Om \nabla u\cdot \nabla v dx = - \int_\Om \abs{\Delta u}^p dx.   \notag
\end{equation}
Combining all of the above in \eqref{pohozaev-2} we get that 
\begin{equation}    \label{int-poho-2}
    \int_\Om \Delta(\abs{\Delta u}^{p-2}\Delta u) (x\cdot\nabla u) dx = -\left(\frac{p-1}{p}\right) \int_{\prt\Om} \abs{\Delta u}^{p} (x\cdot \n) \;d\sigma - \left(\frac{N-2p}{p}\right) \int_\Om \abs{\Delta u}^p dx.   
\end{equation}
We are left with the middle integral in \eqref{pohozaev-1}:
\begin{align}
    \int_\Om \dvg & \left(\frac{\abs{\nabla u}^{p-2}\nabla u}{\delta^p} \right) (x\cdot\nabla u) dx = \int_{\prt\Om} \frac{\abs{\nabla u}^{p-2}}{\delta^p} (x\cdot\nabla u) (\nabla u \cdot \n) \;d\sigma - \int_\Om \frac{\abs{\nabla u}^{p-2}}{\delta^p} \Big(\nabla u\cdot \nabla(x\cdot\nabla u)\Big) dx  \notag\\
    & = - \int_\Om \frac{\abs{\nabla u}^p}{\delta^p}  dx - \frac{1}{2}\int_\Om \frac{\abs{\nabla u}^{p-2}}{\delta^p} \Big(x\cdot \nabla(\abs{\nabla u}^2)\Big) dx,   \notag
\end{align}
where we use again that $\nabla u\cdot \n = \frac{\prt u}{\prt\n}=0$ on $\prt\Om$. Since $u=0$ is constant on the boundary, we know that $\abs{\nabla u} = \abs{\frac{\prt u}{\prt\n}} \abs{\n} = 0$. Hence,
\begin{align}
    \int_\Om \frac{\abs{\nabla u}^{p-2}}{\delta^p(x)} \Big(x\cdot \nabla(\abs{\nabla u}^2)\Big) dx & = - N\int_\Om \frac{\abs{\nabla u}^p}{\delta^p(x)} dx -\frac{p-2}{2} \int_\Om \frac{\abs{\nabla u}^{p-2}}{\delta^p(x)} \Big(x\cdot \nabla(\abs{\nabla u}^2)\Big) dx   \notag\\& + p\int_\Om \frac{\abs{\nabla u}^p}{\delta^{p+1}(x)} (x\cdot\nabla\delta) dx \notag\\
    \Rightarrow \int_\Om \frac{\abs{\nabla u}^{p-2}}{\delta^p(x)} \Big(x\cdot \nabla(\abs{\nabla u}^2)\Big) dx & = - \frac{2N}{p} \int_\Om \frac{\abs{\nabla u}^p}{\delta^p(x)} dx + 2 \int_\Om \frac{\abs{\nabla u}^p}{\delta^{p+1}(x)} (x\cdot\nabla\delta) dx.
\end{align}
Taking into account that $\nabla\delta(x) = -\n(N(x))$, where $N(x)$ is the closest point to $x$ on the boundary (according to Section 3), we are left with 
\begin{align}    \label{int-poho-3}
    \int_\Om \dvg \left(\frac{\abs{\nabla u}^{p-2}\nabla u}{\delta^p} \right) (x\cdot\nabla u) dx & = \left(\frac{N-p}{p}\right) \int_\Om \frac{\abs{\nabla u}^p}{\delta^p} dx + \int_\Om \frac{\abs{\nabla u}^p}{\delta^{p+1}} (x\cdot\n(N(x))) dx.
\end{align}
By substituting \eqref{int-poho-1}-\eqref{int-poho-3} in \eqref{pohozaev-1} we get
\begin{align}
    -&\left(\frac{p-1}{p}\right) \int_{\prt\Om} \abs{\Delta u}^{p} (x\cdot \n) \;d\sigma - \left(\frac{N-2p}{p}\right) \int_\Om \abs{\Delta u}^p dx +\lmb \int_{\prt\Om} \frac{\abs{\nabla u}^{p-2}}{\delta^p} (x\cdot\nabla u) (\nabla u \cdot \n) \;d\sigma  \notag\\
    & - \frac{\lmb}{p}\int_{\prt\Om} \frac{\abs{\nabla u}^p}{\delta^p} (x\cdot \n)\;d\sigma + \lmb\left(\frac{N-p}{p}\right) \int_\Om \frac{\abs{\nabla u}^p}{\delta^p} dx - \lmb\int_\Om \frac{\abs{\nabla u}^p}{\delta^{p+1}} (x\cdot\nabla\delta) dx + \frac{N}{q} \int_\Om |u|^q dx =0
\end{align}
Rearranging, 
\begin{align}  
    \left(\frac{p-1}{p}\right) & \int_{\prt\Om} \abs{\Delta u}^{p} (x\cdot \n) \;d\sigma + \left(\frac{N-2p}{p}\right) \int_\Om \abs{\Delta u}^p dx - \lmb\int_\Om \frac{\abs{\nabla u}^p}{\delta^{p+1}} (x\cdot\n(N(x))) dx \notag\\
    & = \lmb\left(\frac{N-p}{p}\right) \int_\Om \frac{\abs{\nabla u}^p}{\delta^p} dx + \frac{N}{q} \int_\Om |u|^q dx.
\end{align}
\end{proof}
While the regularity needed to rigorously justify the preceding computations is not guaranteed a priori for weak solutions, we can extend the result for the weak formulation via an approximation argument. This is done by starting from the variational formulation of the weak solution. First, we should regularize the test function $x\cdot \nabla u$ by convolution with a $C^\infty-$function and then truncating it near the boundary, allowing us to apply integration by parts. Passing to the limit in the end, we get the result.

\begin{proof}[Proof of Theorem \ref{thmNonExistence}]
Assume there exists $u$ a non-trivial solution of \eqref{ec:systemP}. Multiplying the equation in \eqref{ec:systemP} by $u$ and integrating by parts we get the identity
\begin{equation}   \notag
    \int_\Om \abs{\Delta u}^p dx = \lmb\int_\Om \frac{\abs{\nabla u}^p}{\delta^p} dx + \int_\Om |u|^q dx.
\end{equation}
Substituting in \eqref{pohozaev-0} we get 
\begin{align}    \notag
    \left(\frac{p-1}{p}\right) \int_{\prt\Om} \abs{\Delta u}^{p} (x\cdot \n) \;d\sigma - \lmb\int_\Om \frac{\abs{\nabla u}^p}{\delta^{p+1}} (x\cdot\n(N(x))) dx & =   \lmb \int_\Om \frac{\abs{\nabla u}^p}{\delta^p} dx \notag\\
    & + \left(\frac{N}{q}-\frac{N-2p}{p}\right) \int_\Om |u|^q dx.
\end{align}
Since $\Om$ is the unit ball and $\delta(x)=1-|x|$, then for $x\in \Om\setminus\{0\}$ we have $\n(N(x))= - \nabla\delta(x)=\frac{x}{|x|}$, so $x\cdot \n(N(x)) = |x| \geq 0$. Considering that $\lmb<0$ and $q\geq \frac{N-2p}{Np}$ or $\lmb\leq0$ and $q> \frac{N-2p}{Np}$, we conclude that $u\equiv0$ in $\Om$, which is a contradiction.
\end{proof}

\section{Appendix}
\begin{proof} [Proof of Proposition \ref{prop1}.]
    When $p>2$ notice that, using \eqref{ec:hessiannorm},
    \begin{equation}    \notag
         \|\hess u\|_{L^p}^2 \;=\; \left( \int_{\R^N} \abs{\hess u}^p dx \right)^{\frac{2}{p}} = \left( \int_{\R^N} \abs{\sum_{i,j=1}^N u_{x_ix_j}^2}^{\frac{p}{2}} dx \right)^{\frac{2}{p}}  = \left\| \sum_{i,j=1}^N u_{x_ix_j}^2 \right\|_{L^{\frac{p}{2}}(\R^N)}.
    \end{equation}    
    We apply Minkowski inequality and \eqref{xixj-bound} and we get
    \begin{equation}   \notag
          \|\hess u\|_{L^p}^2 \; \leq \; \sum_{i,j=1}^N \|u_{x_ix_j}^2\|_{L^{\frac{p}{2}}(\R^N)} \; = \; \sum_{i,j=1}^N \|u_{x_ix_j}\|_{L^p(\R^N)}^2   \;\leq\; N^2 \cot^4 \left( \frac{\pi}{2p} \right) \|\Delta u\|_{L^p(\R^N)}^2. 
    \end{equation}
    Raising the last inequality to the power $\frac{p}{2}$, we get \eqref{hess-lap}. For $1<p<2$ we have that 
    \begin{align}  \notag
        \int_{\R^N} \abs{Hu}^p dx & = \int_{\R^N} \abs{\sum_{i,j=1}^N u_{x_ix_j}^2}^\frac{p}{2} dx \leq \int_{\R^N} \sum_{i,j=1}^N u_{x_ix_j}^p = \sum_{i,j=1}^N \|u_{x_ix_j}\|_{L^p(\R^N)}^p    \notag\\
        &\leq N^p \cot^{2p}\left( \frac{\pi}{2p'} \right) \|\Delta u\|_{L^p(\R^N)}^p.
    \end{align}
\end{proof}

\begin{lemma}    \label{lemma-1}
    If $p\geq 2$ it holds
\begin{equation}
    \abs{a+b}^p - \abs{a}^p \leq \frac{p(p-1)}{2}\Big(\abs{a}+\abs{b}\Big)^{p-2} \abs{b}^2 + p \abs{a}^{p-2}a\cdot b\;,\;\; \forall\;a,b\in\R^N.    \label{ec:ineqshafrir1}
\end{equation}
If $1<p<2$ there exists a constant $\gamma=\gamma_p$ such that
\begin{equation}
    \abs{a+b}^p - \abs{a}^p \leq \gamma_p\abs{b}^p + p \abs{a}^{p-2}a\cdot b \;,\;\; \forall a,b\in\R^N.    \label{ec:ineqshafrir2}
\end{equation}
If $0<p\leq1$ it holds
\begin{equation}     \label{ec:ineqshafrir3}
    |a+b|^p - |a|^p \leq |b|^p \;,\;\; \forall\; a,b\in\R^N. 
\end{equation}
\end{lemma}   
\begin{proof}
    The inequalities \eqref{ec:ineqshafrir1} and \eqref{ec:ineqshafrir2} can be found in \cite[Appendix]{shafrir}. For the proof of \eqref{ec:ineqshafrir3}, we apply the first elementary inequality in \eqref{ineqshafrir4}
    \begin{equation}
        (a+b)^p\leq a^p+b^p\;\text{for} \;\;a,b\in\R_+, \;\;0<p<1  \;\;\;\text{and}\;\;\; (a+b)^p\leq 2^{p-1}(a^p+b^p)\;\text{for} \;\;a,b\in\R_+,\;\; p>1.  \label{ineqshafrir4}
    \end{equation}
\end{proof}

We also recall the next useful inequality for vectors in $\R^N$, from \cite[Chapter 10]{lindqvist}.
\begin{lemma}   \label{lemma-3}
    If $p\geq 2$ it holds:
    \begin{equation}    \label{ineqlindqvist>2}
        \left(|a|^{p-2}a-|b|^{p-2}b\right)\cdot(a-b) \geq 2^{2-p}|a-b|^p, \;\; \forall a,b\in \R^N.
    \end{equation}
    If $1<p<2$ it holds:
    \begin{equation}    \label{ineqlindqvist<2}
        \left(|a|^{p-2}a-|b|^{p-2}b\right)\cdot(a-b) \geq (p-1)\frac{\abs{a-b}^2}{(1+\abs{a}^2+\abs{b}^2)^\frac{2-p}{2}}, \;\;\forall a,b\in\R^N.
    \end{equation}
\end{lemma}

\vspace{1cm}
\textbf{Acknowledgements:} The second author was partially supported by a doctoral scholarship provided by the University of Bucharest.


\begin{thebibliography}{} \scriptsize  
\bibitem{adams} R.A. Adams, J.J.F. Fournier; \textit{Sobolev spaces}, 2nd edition, 
Pure and Applied Mathematics, 140, Elsevier/Academic Press, Amsterdam, 2003.
\bibitem{ancona} A. Ancona; \textit{On strong barriers and an inequality of Hardy for domains in $\R^n$}, J. London Math. Soc., vol. 34, no. 2, pg. 274–290, 1986.
\bibitem{avkhadiev1} F.G. Avkhadiev; \textit{Selected results and open problems on Hardy–Rellich and
Poincaré–Friedrichs inequalities}, Anal. Math. Phys., vol. 11, no. 134, 2021.
\bibitem{avkhadiev2} F.G. Avkhadiev, A. Laptev; \textit{Hardy inequalities for nonconvex domains}, Around Research of Vladimir Maz’ya, Int. Math. Ser., 11, Springer, pp. 1–12, 2010.
\bibitem{balinsky1}  A. Balinsky, W. D. Evans, R. T. Lewis; \textit{The Analysis and Geometry of Hardy's Inequality}; Universitext, Springer, 2015.
\bibitem{balinsky2} A. Balinsky; W. D. Evans; R. T. Lewis; \textit{Hardy's inequality and curvature}; J. Funct. Anal.; vol. 262, pg. 648-666, 2012.
\bibitem{banuelos} R. Bañuelos, G. Wang; \textit{Sharp inequalities for martingales with applications to the Beurling–Ahlfors
and Riesz transformations}, Duke Math. J., vol. 80, pg. 575–600, 1995.
\bibitem{barbatis3} G. Barbatis; \textit{Best constants for higher-order Rellich inequalities in $L p(\Om)$}, Math. Z., vol. 255, pg. 877–896, 2007.
\bibitem{barbatis2} G. Barbatis, S. Filippas, A.Tertikas; \textit{A unified approach to improved $L^p$ Hardy inequalities with best constants}, Trans. Amer. Math. Soc., vol. 356, no. 6, pg. 2169-2196, 2003.
\bibitem{barbatis} G. Barbatis, A. Tertikas; \textit{On a class of Rellich inequalities}, J. of Comp. and Appl. Mathematics, vol. 194, no. 1, 2005, DOI: 10.1016/j.cam.2005.06.020.
\bibitem{bhakta} M. Bhakta; \textit{Entire solutions for a class of elliptic equations involving
$p$-biharmonic operator and Rellich potentials}, J. Math. Anal. Appl., vol 423, pg. 1570-1579, 2015.
\bibitem{bowman} F. Bowman; \textit{Introduction to Bessel functions}, Dover Publ., Inc., New York, 1958.
\bibitem{brezis-carte} H. Brezis; \textit{Funct. Anal., Sobolev spaces and partial differential equations}, Universitext, Springer, New York, 2011.
\bibitem{brezis-lieb} H. Brezis, E. Lieb; \textit{A relation between pointwise convergence of functions and convergence of functionals}, Proc. Amer. Math. Soc., vol. 88, no. 3, pg. 486-490, 1983.
\bibitem{brezis} H. Brezis, M. Marcus; \textit{Hardy's inequalities revisited}; Ann. Sc. Norm. Super. Pisa, Cl. Sci. $(4)$, vol. 25, no. 1-2, pg. 217-237, 1997.
\bibitem{cassano} B. Cassano, L. Cossetti, L. Fanelli; \textit{Improved Hardy-Rellich inequalities}, Comm. Pure Appl. Math., vol. 21, iss. 3, pg. 867-889, 2022.
\bibitem{cassano2} B. Cassano; V. Franceschi; D. Krejčiřík; D. Prandi; \textit{Horizontal magnetic fields and improved Hardy inequalities in the Heisenberg group}, Comm. Partial Differential Equations, vol. 48, no. 5, pg. 711–752, 2023.
\bibitem{cazacuHR} C. Cazacu; \textit{A new proof of the Hardy-Rellich inequality in any dimension}; Proc. Roy. Soc. Edinburgh Sect. A, 150, no. 6, pg. 2894–2904, 2020.
\bibitem{cazacu} C. Cazacu, \textit{The method of super-solutions in Hardy and Rellich type inequalities in the $L^2$ setting: an overview of well-known results and short proofs}; Rev. Roumaine Math. Pures Appl. 66, no. 3-4, pg. 617–638, 2021.
\bibitem{cazacu-boundary} C. Cazacu; \textit{Schrodinger operators with boundary singularities: Hardy inequality, Pohozaev identity and controllability results}, J. Funct. Anal., vol. 263, no. 12, pg. 3741–3783, 2012.
\bibitem{davies-hinz} E. Davies, A. Hinz; \textit{Explicit constants for Rellich inequalities in $L^p(\Om)$}, Math. Z. 227, pg. 511–523, 1998.
\bibitem{drissi} A. Drissi, A. Ghanmi, D. Repovs; \textit{Singular $p$-Biharmonic Problems Involving the Hardy-Sobolev Exponent}, Electron. J. Differential Equations, vol. 2023, no. 61, pg. 1-12, 2023.
\bibitem{edmunds} D. E. Edmunds, W. D. Evans; \textit{The Rellich inequality}, Rev. Mat. Complut., vol. 29, no. 3, pg. 511–530, 2016.
\bibitem{evansCarte} L.C. Evans; \textit{Partial differential equations}, Graduate Studies in Mathematics, vol. 19, Amer. Math. Soc., Providence, RI, 1998.
\bibitem{evans} W.D. Evans, D.J. Harris; \textit{Sobolev embeddings for generalized ridged domains}, Proc. Lond. Math. Soc., vol. 54, pg. 141–175, 1987.
\bibitem{mmfall} M. M. Fall; \textit{A note on Hardy’s inequalities with boundary singularities}, Nonlinear Anal., vol. 75, no. 2, pg. 951-963, 2010.
\bibitem{fillipas1} S. Filippas, V. Maz’ya, A. Tertikas; \textit{On a question of Brezis and Marcus}, Calc. Var. Partial Differential Equations, vol. 25, pg. 491–501, 2006.
\bibitem{fremlin} D. Fremlin; \textit{Skeletons and central sets}, Proc. Lond. Math. Soc., vol. 74, iss. 3, pg. 701–720, 1997.
\bibitem{federer} H. Federer; \textit{Curvature measures}, Trans. Amer. Math. Soc., vol. 93, no. 3, pg. 418–491, 1959.
\bibitem{azorero-alonso} J. P. Garcia Azorero, I. Peral Alonso, \textit{Hardy Inequalities and Some Critical Elliptic and Parabolic Problems}, J. Differential Equations, 144, pg. 441-476, 1998.
\bibitem{gilbardtrud}  D. Gilbarg, N.S. Trudinger; \textit{Elliptic Partial Differential Equations of Second Order}, Springer-Verlag, 2001.
\bibitem{mihailescu} A. Grecu, M. Mihăilescu; \textit{The torsion problem of the $p$-Bilaplacian}, Nonlinear Anal. Real World Appl., vol. 79, 2024.
\bibitem{pigola} B. Güneysu, S. Pigola; \textit{The Calderón–Zygmund inequality and Sobolev spaces on noncompact Riemannian manifolds}, Adv. Math., vol. 281, pg. 353–393, 2015.
\bibitem{hardy} G. Hardy, J. E. Littlewood, G. Polya; \textit{Inequalities}, Cambridge Univ. Press,
Cambridge, UK, 1934.
\bibitem{hoffman} M.Hoffmann–Ostenhof,T.Hoffmann–Ostenhof, A.Laptev; \textit{A geometrical version of Hardy’s inequality}, J. Funct. Anal., vol. 189, pg. 539–548, 2002.
\bibitem{iwaniec}  T. Iwaniec, G. Martin; \textit{Riesz transforms and related singular integrals}, J. Reine Angew. Math., vol. 473, pg. 25–57, 1996.
\bibitem{kombe} I. Kombe, M. Ozaydin; \textit{Improved Hardy and Rellich inequalities on Riemannian manifolds}, Trans. Amer. Math. Soc., vol. 361, no. 12, pg. 6191–6203, 2009.
\bibitem{opic} A. Kufner, B. Opic; \textit{Hardy-type Inequalities}, Pitman Research Notes in Math., vol. 219, Longman 1990. 
\bibitem{lam1} N. Lam; \textit{Hardy and Hardy–Rellich type inequalities with Bessel pairs}, Ann. Acad. Sci. Fenn. Math., vol. 43, pg. 211–223, 2018.
\bibitem{lam2} N. Lam, G. Lu, L. Zhang; \textit{Geometric Hardy’s inequalities with general distance functions}, J. Funct. Anal., vol. 279, iss. 8, 2020.
\bibitem{lewis} R.T. Lewis, J. Li, Y. Li; \textit{A geometric characterization of a sharp Hardy inequality}, J. of Funct. Anal., vol. 262, pg.3159-3185, 2012.
\bibitem{lindqvist} P. Lindqvist; {\it Notes on the $p$-Laplace equation}, Report. University of Jyv\"askyl\"a{} Department of Mathematics and Statistics, 102, Univ. Jyv\"askyl\"a, Jyv\"askyl\"a, 2006.
\bibitem{sobolevskii} T. Matskewich, P. E. Sobolevskii; \textit{The best possible constant in generalized Hardy's inequality for convex domain in $\R^n$}; Nonlinear Anal., Theory, Methods \& Appl., vol. 28, no. 8, pg. 1601-1610, 1997. 
\bibitem{pinchover} M. Marcus, V. Mizel, Y. Pinchover; \textit{On the best constant for Hardy's inequality in $\R^n$}, Trans. Amer. Math. Soc., vol. 350, no. 8, pg. 3237–3255, 1998.
\bibitem{mantegazza} C. Mantegazza, A.C. Mennucci; \textit{Hamilton-Jacobi equations and distance functions on
Riemannian manifolds}, Appl. Math. Optim., vol. 47, pg. 1–25, 2003.
\bibitem{mawhin} J. Mawhin, M. Willem; \textit{Origin and evolution of the Palais–Smale condition in critical point theory}, J. Fixed Point Theory Appl., vol. 7, pg. 265–290, 2010.
\bibitem{metafune} G. Metafune, M. Sobajima, C. Spina; \textit{Weighted Calderon-Zygmund and Rellich inequalities in $L^p$}, Math. Ann. vol. 361, pg. 313-366, 2015.
\bibitem{mitidieri} E. Mitidieri; \textit{A Rellich type identity and Applications}, Comm. Partial Differential Equations, vol. 18, no. 1-2, pg. 125–151, 1993.
\bibitem{owen} M.P. Owen; \textit{The Hardy-Rellich inequality for polyharmonic operators}, Proc. Roy. Soc. Edinburgh Sect. A, vol. 129A, pg. 825–839, 1999.
\bibitem{peral-monograph} I. Peral, F. Soria; \textit{Elliptic and Parabolic Equations Involving the Hardy-Leray Potential}, De Gruyter Series in Nonlinear Analysis and Appl., vol. 38, 2021. 
\bibitem{pichorides}  S.K. Pichorides; \textit{On the best values of the constants in the theorems of M. Riesz, Zygmund and
Kolmogorov}, Studia Math., vol. 44, pg. 165–179, 1972.
\bibitem{rellich} F. Rellich; \textit{Halbeschränkte Differentialoperatoren höherer Ordnung}, Proceedings of the International Congress of Mathematicians 1954, vol. III, pg.243–250, 1956. 
\bibitem{shafrir}  I. Shafrir, \textit{Asymptotic behaviour of minimizing sequences for Hardy’s inequality}. Commun. Contemp. Math., vol. 2, no. 2, pg. 151–189, 2000.
\bibitem{tertikas} A. Tertikas, N.B. Zographopoulos; \textit{Best constants in the Hardy-Rellich inequalities and related improvements}; Adv. Math, vol. 209, no. 2, pg. 407-459, 2007.
\bibitem{tidblom} J. Tidblom; \textit{$L^p$ Hardy inequalities in general domains}, PhD thesis, Department of Mathematics, Stockholm University, 2003.
\bibitem{wang} H. Xie, J. Wang; \textit{Infinitely many solutions for $p$-harmonic equation with singular term}, J. Inequal. Appl., vol. 2013, no. 9, 2013.
\end{thebibliography}
\end{document}